\documentclass[11pt,oneside]{amsart}
\usepackage[utf8]{inputenc}
\usepackage{amsthm}
\usepackage{amssymb}
\usepackage{float}
\usepackage{geometry}
\usepackage{tikz-cd}
\usepackage[hidelinks]{hyperref}
\usepackage{todonotes}


\numberwithin{equation}{section}

\floatstyle{ruled}
\newfloat{algorithm}{tbp}{loa}
\floatname{algorithm}{Algorithm}

\theoremstyle{plain}
\newtheorem{thm}{Theorem}[section]
\newtheorem{lem}[thm]{Lemma}
\newtheorem{prop}[thm]{Proposition}
\newtheorem{cor}[thm]{Corollary}

\theoremstyle{definition}
\newtheorem{defn}[thm]{Definition}

\theoremstyle{remark}
\newtheorem{rem}[thm]{Remark}

\newcommand{\bigmid}{\mathrel{}\middle|\mathrel{}}

\newcommand{\CC}{\mathbb{C}}

\newcommand{\GG}{\mathbb{G}}

\newcommand{\NN}{\mathbb{N}}
\newcommand{\PP}{\mathbb{P}}
\newcommand{\QQ}{\mathbb{Q}}
\newcommand{\RR}{\mathbb{R}}
\newcommand{\ZZ}{\mathbb{Z}}

\newcommand{\ba}{\mathbf{a}}

\newcommand{\bc}{\mathbf{c}}
\newcommand{\bd}{\mathbf{d}}
\newcommand{\be}{\mathbf{e}}
\renewcommand{\bf}{\mathbf{f}}
\newcommand{\bg}{\mathbf{g}}
\newcommand{\bh}{\mathbf{h}}
\newcommand{\bi}{\mathbf{i}}
\newcommand{\bj}{\mathbf{j}}

\newcommand{\br}{\mathbf{r}}
\newcommand{\bs}{\mathbf{s}}

\newcommand{\bu}{\mathbf{u}}
\newcommand{\bv}{\mathbf{v}}
\newcommand{\bw}{\mathbf{w}}
\newcommand{\bx}{\boldsymbol{x}}
\newcommand{\by}{\boldsymbol{y}}

\newcommand{\cA}{\mathcal{A}}

\newcommand{\cF}{\mathcal{F}}

\newcommand{\cJ}{\mathcal{J}}

\newcommand{\cL}{\mathcal{L}}
\newcommand{\cM}{\mathcal{M}}
\newcommand{\cN}{\mathcal{N}}
\newcommand{\cO}{\mathcal{O}}

\newcommand{\cQ}{\mathcal{Q}}

\newcommand{\cS}{\mathcal{S}}
\newcommand{\cT}{\mathcal{T}}

\newcommand{\fp}{\mathfrak{p}}
\newcommand{\fq}{\mathfrak{q}}

\newcommand{\ft}{\mathfrak{t}}

\newcommand{\fM}{\mathfrak{M}}

\renewcommand{\rm}{\mathrm{m}}

\newcommand{\balpha}{\boldsymbol{\alpha}}
\newcommand{\bbeta}{\boldsymbol{\beta}}

\newcommand{\biProj}{\operatorname{\mathbf{Proj}}}

\newcommand{\bmax}{\mathbf{max}}
\newcommand{\bmin}{\mathbf{min}}
\newcommand{\bone}{\mathbf{1}}

\newcommand{\bzero}{\mathbf{0}}

\newcommand{\Cl}{\operatorname{Cl}}

\newcommand{\Cone}{\operatorname{Cone}}

\newcommand{\cProj}{\mathcal{P}roj}
\newcommand{\cSpec}{\mathcal{S}pec}

\newcommand{\Exc}{\mathrm{Exc}}
\newcommand{\Ext}{\operatorname{Ext}}

\newcommand{\Hom}{\operatorname{Hom}}

\newcommand{\lcm}{\operatorname{\mathrm{lcm}}}

\newcommand{\NE}{\overline{\mathrm{NE}}}

\newcommand{\pr}{\operatorname{pr}}
\newcommand{\Proj}{\operatorname{Proj}}
\newcommand{\Qbar}{\overline{\QQ}}

\newcommand{\Spec}{\operatorname{Spec}}

\begin{document}
\title{An algorithm for the minimal model program in dimension three}
\author{Takehiko Yasuda}
\address{Department of Mathematics, Graduate School of Science, the University
of Osaka, Toyonaka, Osaka 560-0043, JAPAN}
\address{Kavli Institute for the Physics and Mathematics of the Universe, The
University of Tokyo, 5-1-5 Kashiwanoha, Kashiwa, Chiba, 277-8583,
Japan}
\email{yasuda.takehiko.sci@osaka-u.ac.jp}
\begin{abstract}
We construct an algorithm for the minimal model program in dimension
three over a computable field of characteristic zero admitting a
splitting algorithm, such as number fields and the field of algebraic numbers. As
auxiliary results, we also construct algorithms for computing bigraded
global Hom modules and for computing Stein factorization.
\end{abstract}

\maketitle

\tableofcontents

\section{Introduction}

The minimal model program occupies a central position in birational
geometry. Starting from a projective variety with mild singularities,
it performs a sequence of birational transforms, and if it terminates,
produces either a minimal model or a Mori fiber space. The program
is well established for threefolds in characteristic zero. For more
details on the program, we refer the reader to the books \cite{kollar1998birational,matsuki2002introduction,hacon2010classification,fujino2017foundations,kawakita2023complex,kawamata2024algebraic}.
As is described with flowcharts (see \cite[Figure 2]{kawamata1987introduction},
\cite[Flowchart 3-1-15]{matsuki2002introduction}), the minimal model
program is sometimes regarded as being algorithmic. However, some
steps of the program rely on non-constructive existence theorems and
there has been no genuine algorithm for the program in the literature,
to the best of the author's knowledge. In special cases, Lazić and
Schreyer carried out computations closely related to the minimal model
program using a computer algebra system \cite{lazic2022birational}.
Strategies towards an algorithmic approach were proposed by the author
\cite[Section 10.4]{yasuda2023theisomorphism} and by Lazić \cite{lazic2024programming}.

The purpose of this paper is to present an algorithm for a version of the minimal
model program in dimension three. Our main result is stated in more
precise terms as follows:
\begin{thm}
\label{thm:main}
There is an algorithm, Algorithm \ref{algo:MMP}, to run a version of the minimal
model program in dimension three over a field \(k\) of characteristic
zero that is computable and admits a splitting algorithm; when
a normal projective variety \(X\) of dimension 3 with only log terminal
singularities is given as an input, the algorithm returns, as an
output, a finite sequence of rational maps
\[
X=X_{0}\overset{f_{1}}{\dasharrow}X_{1}\overset{f_{2}}{\dasharrow}\cdots\overset{f_{n}}{\dasharrow}X_{n},
\]
with the following
properties:
If \(X_{i-1}\) is \(\QQ\)-Gorenstein, then \(f_i\) is the contraction of a \(K\)-negative extremal face.
If \(X_{i-1}\) is not \(\QQ\)-Gorenstein, then \(X_i\) is the relative
canonical model of \(X_{i-1}\), with
\(f_i^{-1}\colon X_i\to X_{i-1}\) the natural morphism.
The sequence ends in one of two
ways: either \(f_{n}\) is birational as well and \(K_{X_{n}}\) is nef, or \(f_{n}\) is a
\(K\)-negative fibration with \(\dim X_n <3\).
\end{thm}

There is an important difference between the MMP in this
theorem and the usual/classical MMP.
Firstly, we do not assume that the input variety is \(\QQ\)-factorial.
Secondly, the contracted extremal faces can have arbitrary dimension,
and the contractions can have relative Picard number greater than one.
We call such a contraction a \(K\)-negative fibration when it is not
birational; it need not be a Mori fiber space. If we omit from the
output the final \(K\)-negative fibration, if any, and every
non-\(\QQ\)-Gorenstein target \(W\) of a birational contraction, then the resulting sequence is a \(K\)-MMP
in the sense of \cite[Definition~3.5 and Remark~3.6]{hashizume2025normal}.
Moreover, with the chosen ample divisor, each step is a one-step
\(K\)-MMP with scaling in Hashizume's sense and a step with scaling in
Kollár's sense \cite[Definition~1]{kollar2021relative}.

Note also that our base field is not \(\CC\): its elements must be represented by finite data, with computable arithmetic and decidable equality, and polynomials in one variable must be factored algorithmically. The first condition makes \(k\) a computable field, while the second is the existence of a splitting algorithm; the latter is strictly stronger.
These conditions support Gröbner-basis and irreducible-component computations, and the splitting algorithm is inherited by the extensions arising in primary decomposition, whose algebraic steps are separable because \(k\) has characteristic zero.
They hold for \(\QQ\), its finitely generated extensions, and \(\Qbar\).
We refer the reader to \cite[\S1 and \S2.2]{harrisontrainor2017computable}, where both notions and all of these facts are collected.

In our algorithm, a projective variety over \(k\) is represented as the Proj scheme,
\(\Proj R\), of an explicitly described graded ring
\(R=k[x_{0},\dots,x_{m}]/\langle g_{1},\dots,g_{l}\rangle\).
Here the polynomial ring \(k[x_{0},\dots,x_{m}]\) is graded by specifying
degrees of variables \(x_{0},\dots,x_{m}\) and \(g_{1},\dots,g_{l}\)
are homogeneous polynomials. In particular, we allow the ring \(R\)
not to be standard graded (that is, generated in degree one). This
simplifies the algorithm, since graded rings constructed in the algorithm
are not a priori standard graded.\footnote{It should be noted that any data given in a non-standard graded framework
can be transformed into data in a standard graded framework.
In particular, the varieties \(X_{i}\) obtained as an output of our algorithm can be
described as closed subvarieties of projective spaces by specifying
their defining homogeneous polynomials.}
We use such finite data to represent projective varieties appearing
in the input and output of our algorithm.
For each rational map \(f_{i}\)
in the sequence in the theorem, either
\begin{enumerate}
\item \(f_{i}\) is a morphism, or
\item \(f_{i}\) is birational and \(f^{-1}_{i}\) is a morphism.
\end{enumerate}
A morphism \(f\colon Z\to Y\) of projective varieties is determined
by its graph \(\Gamma_{f}\subset Z\times Y\), which is again a projective
variety.
If we write \(Z=\Proj S\) and \(Y=\Proj R\) with \(S\) and \(R\)
graded rings, then \(\Gamma_{f}\) is most naturally represented
by a bihomogeneous ideal of the bigraded ring \(S\otimes R\).
For each
rational map \(f_{i}\) in an output, either \(f_{i}\) or \(f^{-1}_{i}\)
is represented by bihomogeneous generators of a bihomogeneous ideal.
For this reason, we use bigraded rings as one of the essential data types.

A technical issue in constructing the algorithm is the lack of an
unconditional procedure for computing the Picard number. Assuming
the Tate conjecture, it can be computed using the algorithm of Poonen,
Testa, and van Luijk \cite{poonen2015computing}, but the conjecture
remains open. The Mori cone plays an essential role in the minimal
model program and is a cone in a real vector space of dimension equal
to the Picard number. To stay unconditional, our algorithm computes
neither the Picard number nor the Mori cone, producing each contraction from a canonically produced semi-ample divisor.
On each \(\QQ\)-Gorenstein model \(X_i\), we compute an ample Cartier divisor
\(H_i\) and the nef threshold \(\lambda_i\); the semi-ample divisor
\(K_{X_i}+\lambda_iH_i\) then gives, after Stein factorization, the
contraction of a \(K_{X_i}\)-negative extremal face. If we get a non-\(\QQ\)-Gorenstein birational model, we compute its relative canonical model, which is the counterpart of a flip in our setting. 
Thus every birational iteration is a step of the MMP with scaling of
\(H_i\), but the resulting birational sequence is not a single MMP with
scaling, since we take a new ample divisor on each model. This choice is
made for computational reasons: we choose \(H_i\) to be compatible with
the monograded presentation of \(X_i\); see Remark
\ref{rem:choice-ample-divisor}.

From the theoretical viewpoint, the hardest part of our construction
is an algorithm for Stein factorization. A contraction morphism is
constructed by applying Stein factorization to the morphism associated
to some linear system. To compute a homogeneous coordinate ring of
an intermediate variety produced by Stein factorization, we need a
bigraded version of Smith's result on computation of global Hom modules
\cite{smith2000computing}. Algorithms for Stein factorization and
bigraded global Hom modules would be of independent interest.

The reason why we focus on dimension three is that we need the
termination of flips and the abundance conjecture, which are yet to be
established in dimension \(\ge4\). Termination enters through Proposition
\ref{prop:termination}: the birational steps form an MMP for pairs that
need not be \(\QQ\)-factorial, and their termination in dimension three
rests on termination of ordinary flips there. Abundance enters once, in
the test of whether \(K_{X_{i}}\) is nef.

The theorems of the minimal model program that we use are stated over
\(\CC\), or over an algebraically closed field of characteristic zero in the literature. They
apply to our base field \(k\) as follows. Since \(k\) is countable, so is its algebraic closure \(\overline{k}\). Therefore, these fields
embed into \(\CC\).
For a normal variety \(X\) over \(k\), the base change \(X_{\overline{k}}\) consists of finitely many connected components, which are pairwise isomorphic normal varieties.
Every property we invoke of a variety or of a morphism --- normality, log
terminality, nefness, ampleness, base point freeness, global generation,
smallness, finite generation of a sheaf of graded algebras --- is
preserved and reflected by extension of the base field and can be read
off one component at a time, and every object we construct ---
\(\omega^{[n]}_{X}\), the threshold \(\lambda\), a Stein factorization, a
relative \(\Proj\) --- commutes with base field extension. A statement of
either kind may therefore be verified after base change to
\(\overline{k}\) and then to \(\CC\) and its conclusion carried back to \(k\),
and we do so without further comment. 
Termination of flips and abundance, the two theorems that restrict us to
dimension three, also descend to \(k\).

As a proof of concept, prototype Macaulay2 implementations have been
written for Stein factorization, for computing relative canonical
models, and for driving the MMP algorithm. Each main computational
step has been run on explicit examples; the code, examples, and current
limitations are summarized in Remark \ref{rem:implementation}.

\subsection*{Comparison with the previous version}

An earlier version of this paper asked its input to be \(\QQ\)-factorial
and returned an MMP sequence in the classical sense, every contraction being of
relative Picard number one and \(\QQ\)-factoriality being carried through the program except the targets of small contractions.  
The base field was assumed to be \(\Qbar\).
The algorithm presented there was far more expensive than the present
one. It contained exhaustive searches for curves and divisors as well
as computations of Betti numbers for singular homology. Because of
these features, it seemed to be impossible to implement the previous
algorithm in practice. The present algorithm can be implemented and
each step can be run for simple examples, though further improvement
in efficiency is still needed to get an interesting output for
complicated examples.

\subsection*{Convention}

Throughout the paper, we work over a fixed field \(k\) as above, of
characteristic zero, computable and with a splitting algorithm. A variety means an
integral projective scheme over \(k\), unless otherwise noted. Every ring is
assumed to be equipped with a structure of \(k\)-algebra. Every morphism
of varieties as well as every ring homomorphism is assumed to be a
morphism/homomorphism over \(k\). A product of varieties means a fiber
product over \(k\). A tensor product of rings and modules is also
defined over \(k\).
By \(\NN\), we mean the set of nonnegative integers,
while we denote by \(\ZZ_{>0}\) the set of positive integers. 

\subsection*{Acknowledgements}

The author thanks Osamu Fujino, Kenta Hashizume, Yujiro
Kawamata, Masanori Kobayashi, Vladimir Lazić, Joaquin Moraga, Yusuke
Nakamura, and Shota Okazawa for helpful conversations.
He is
especially indebted to Makoto Enokizono for bringing to his
attention a version of the minimal model program that contracts
faces.

The first two arXiv versions were written by the author, with large
language models used only for English editing. In preparing the
third version, the author also used them to assist with literature
searches, discuss proofs, and draft portions of the text. The author
independently checked all material produced with their assistance.
Claude (Anthropic) and ChatGPT (OpenAI) were used for these purposes.
The Macaulay2 implementation was also developed with the help of
large language models, to a much greater extent, as explained in
Remark \ref{rem:implementation}.

This work was supported by JSPS KAKENHI Grant
Numbers JP21H04994, JP23K25767, and JP24K00519.

\section{Monograded varieties and bigraded varieties}

We assume that every variety is irreducible, reduced and projective
over \(k\), unless otherwise noted. When we discuss algorithms, varieties
and their morphisms are endowed with data representing them. We explain
below types of data that we will use.

\subsection{Monograded varieties\label{subsec:monograded-bigraded}}

Our primary data type to represent varieties is a collection of finitely
many weighted homogeneous polynomials that define a closed subvariety
of a weighted projective space. Let \(\bc=(c_{0},\dots,c_{m})\in\ZZ^{m+1}_{>0}\)
and let \(k[\bx]:=k[x_{0},\dots,x_{m}]\) be a polynomial ring given
with an \(\NN\)-grading by \(\deg(x_{i})=c_{i}\). Its Proj scheme, \(\Proj k[\bx]\),
is the weighted projective space \(\PP(\bc)=\PP(c_{0},\dots,c_{m})\).
A closed subscheme of \(\PP(\bc)\) is written as \(V(\bf)=V(f_{1},\dots,f_{l})\),
where \(\bf=(f_{1},\dots,f_{l})\in k[\bx]^{l}\) is a tuple of homogeneous
polynomials relative to the grading defined above such that the induced
ideal \(\langle\bf\rangle=\langle f_{1},\dots,f_{l}\rangle\) does not contain
the irrelevant ideal \(k[\bx]_{\dagger}:=\langle x_{0},\dots,x_{m}\rangle\).
It is a closed subvariety when \(\langle\bf\rangle\) is prime.
\begin{defn}
\label{defn:monograded-scheme}
A \emph{monograded scheme} means a closed subscheme \(V(\bf)\subset\PP(\bc)\)
endowed with the datum of the defining homogeneous polynomials
\(\bf=(f_{1},\dots,f_{l})\). It is a \emph{monograded variety} when
\(\langle\bf\rangle\) is prime.
\end{defn}

Giving a monograded scheme is equivalent to giving an explicit presentation
\(R=k[\bx]/\langle\bf\rangle\) of its homogeneous coordinate ring \(R\).
The term ``monograded'' is due to the fact that the homogeneous
coordinate ring is graded by \(\ZZ\), a free abelian group of rank
one. (In general, we call rings and modules graded by \(\ZZ\) \emph{monograded},
and call ones graded by \(\ZZ^{2}\) \emph{bigraded}. We do not consider
gradings of ranks higher than two.) If we write \(R=\bigoplus_{i\in\ZZ}R_{i}\),
then \(R_{i}\) denotes the degree-\(i\) part of \(R\). By definition,
we have \(R_{0}=k\). The homogeneous ideal \(R_{\dagger}:=\bigoplus_{i>0}R_{i}\)
is called the \emph{irrelevant ideal}. The underlying set of \(\Proj R\)
is identified with the set of homogeneous prime ideals \(\fp\subset R\)
with \(\fp\nsupseteq R_{\dagger}\).

\subsection{Ambient toric varieties}
As a secondary data type to represent varieties, we also consider
bigraded varieties. Let \(k[\by]:=k[y_{0},\dots,y_{n}]\) be another
polynomial ring with \(d_{j}:=\deg(y_{j})\in\ZZ_{>0}\), and assume
\(n,m\ge1\). We grade the tensor product
\(k[\by]\otimes k[\bx]=k[\by,\bx]\) by \(\NN^{2}\), setting
\[
\deg(y_{j})=(d_{j},0)\text{ and }\deg(x_{i})=(a_{i},c_{i}),
\]
where \(\ba=(a_{0},\dots,a_{m})\in\NN^{m+1}\) is a further datum. 

Write \(\bd=(d_{0},\dots,d_{n})\) and \(\bc=(c_{0},\dots,c_{m})\), let
\(L\subset\ZZ^{n+1}\oplus\ZZ^{m+1}\) be the saturation of the subgroup
generated by \((\bd,\ba)\) and \((\bzero,\bc)\), and put
\[
N:=\bigl(\ZZ^{n+1}\oplus\ZZ^{m+1}\bigr)/L,
\]
a lattice of rank \(n+m\). Let \(\pi\) denote the quotient map \(\ZZ^{n+1}\oplus \ZZ^{m+1}\to N\), let
\[
u_{j}:=\pi(e_{j})\ (0\le j\le n),\qquad v_{i}:=\pi(f_{i})\ (0\le i\le m)
\]
be the images of the standard bases of the two summands, and let
\(\Sigma\) be the collection of the cones
\[
\sigma_{j,i}:=\Cone\bigl(u_{j'}\ (j'\ne j),\ v_{i'}\ (i'\ne i)\bigr)
\qquad(0\le j\le n,\ 0\le i\le m)
\]
in \(N_\RR := N\otimes \RR \) together with all their faces. The kernel of \(\pi \otimes \RR\) is spanned
by \((\bd,\ba)\) and \((\bzero,\bc)\), so that
\begin{equation}\label{eq:fan-relations}
\sum_{j}d_{j}u_{j}+\sum_{i}a_{i}v_{i}=0,\qquad\sum_{i}c_{i}v_{i}=0
\end{equation}
form a basis for all relations among the \(u_{j}\) and \(v_{i}\).
\begin{lem}
\label{lem:twisted-fan}
\(\Sigma\) is a complete simplicial fan in \(N_\RR\) whose rays
are exactly the \(u_{j}\) and the \(v_{i}\).
\end{lem}

\begin{proof}
Use \(\pi\) also for the induced
surjection \(\pi\otimes \RR \colon \RR^{n+1}\oplus\RR^{m+1}\to N_{\RR}\). Let
\[
\tau_{j,i}:=\Cone\bigl(e_{j'}\ (j'\ne j),\ f_{i'}\ (i'\ne i)\bigr)
\qquad(0\le j\le n,\ 0\le i\le m)
\]
and let \(\Sigma'\) be the simplicial fan in \(\RR^{n+1}\oplus\RR^{m+1}\) formed by all the faces of
these cones. Its support \(|\Sigma'|\) consists of the points where all coordinates are nonnegative, at least one coordinate is zero in the first $n+1$ entries, and the same holds for the last $m+1$ entries.

We claim that \(\pi\) restricts to a bijection \(|\Sigma'|\to N_{\RR}\).
This is equivalent to saying that for every
\((\bs,\br)\in\RR^{n+1}\oplus\RR^{m+1}\) there exists a unique
\((\alpha,\beta)\in\RR^{2}\) such that
\[
(\bs,\br)+\alpha(\bd,\ba)+\beta(\bzero,\bc)\in|\Sigma'|.
\]
That is so, and the pair is explicitly determined by
\begin{align*}
  \alpha &=\min\{\alpha'\mid\bs+\alpha'\bd\in\RR_{\ge0}^{n+1}\},\\
  \beta &=\min\{\beta'\mid\br+\alpha\ba+\beta'\bc\in\RR_{\ge0}^{m+1}\}.
\end{align*}

The collection \(\Sigma\) consists of the images of the cones in
\(\Sigma'\). Being injective on \(|\Sigma'|\), the map \(\pi\) carries a cone
of \(\Sigma'\) to a strongly convex simplicial cone of the same
dimension, and the intersection of two of them to the intersection of
their images, so \(\Sigma\) is again a simplicial fan, its maximal cones
have dimension \(n+m\), and its rays are the images of the rays of
\(\Sigma'\), that is, the \(u_{j}\) and the \(v_{i}\). Since \(|\Sigma|=\pi(|\Sigma'|)=N_{\RR}\), the fan \(\Sigma\) is complete.
\end{proof}

\begin{defn}
  We write \(\biProj k[\by,\bx]\) for the complete \(\QQ\)-factorial toric variety associated to \(\Sigma\).
\end{defn}

This notation is justified by the multigraded Proj construction as follows.
We take as the irrelevant ideal of the bigraded ring \(k[\by,\bx]\)
\[
k[\by,\bx]_{\dagger}:=\langle y_{j}x_{i}\mid0\le j\le n,\,0\le i\le m\rangle,
\]
and glue the affine schemes
\[
\Spec\bigl((k[\by,\bx]_{y_{j}x_{i}})_{\bzero}\bigr)
\qquad(0\le j\le n,\ 0\le i\le m).
\]
Here the subscript \(\bzero\) denotes the degree-\((0,0)\) part of the bigraded ring \(k[\by,\bx]_{y_{j}x_{i}}\).
This applies the gluing procedure for multigraded Proj
\cite[Construction 3.1]{mayeux2025onmultigraded} to the selected relevant
elements \(y_jx_i\).  Here we do not mean the full Brenner--Schr\"oer
multihomogeneous spectrum, which uses all relevant elements
\cite[Section~2]{brenner2003amplefamilies} and can be larger and
nonseparated; the distinction is discussed below.  In the Cox construction
for \(\Sigma\), the monomial complementary to the maximal cone
\(\sigma_{j,i}\) is \(y_{j}x_{i}\) \cite[\S5.1]{cox2011toric}, and there is
a natural identification
\[
\bigl(k[\by,\bx]_{y_{j}x_{i}}\bigr)_{\bzero}
 \cong k[\sigma_{j,i}^{\vee}\cap N^{\vee}].
\]
Thus these are precisely the affine toric charts of \(\Sigma\), with the
same overlap maps, so their gluing is the toric variety \(X_{\Sigma}\).
For \(\ba=\bzero\) the first
relation of \eqref{eq:fan-relations} involves only the \(u_{j}\), so
\(N\) splits and \(\Sigma\) is a product of fans so that
\(\biProj k[\by,\bx]=\PP(\bd)\times\PP(\bc)\). In general the variety is
proper over \(\PP(\bd)\) with fiber \(\PP(\bc)\), and it is a product only
when the ratios \(a_{i}/c_{i}\) are all equal.

The Cox construction identifies \(\biProj k[\by,\bx]\) with the
geometric quotient
\[
\biProj k[\by,\bx]
 \cong
\bigl(\Spec k[\by,\bx]\setminus
V(k[\by,\bx]_{\dagger})\bigr)/\GG_m^2
\]
\cite[Theorem 5.1.11]{cox2011toric}, where the action is determined by the bigrading on \(k[\by,\bx]\).

Since \(\Sigma\) is simplicial, a torus-invariant \(\QQ\)-divisor is
\(\QQ\)-Cartier and corresponds to its support function \(\psi\), which is
determined by the values \(\psi(u_{j})\) and \(\psi(v_{i})\). Consider the
exact sequence derived from the definition of \(N\),
\[
0\to L_{\QQ}\to\QQ^{n+1}\oplus\QQ^{m+1}\to N_{\QQ}\to0,
\]
and its dual sequence
\[
0\to(N_{\QQ})^{\vee}\to\QQ^{n+1}\oplus\QQ^{m+1}\to(L_{\QQ})^{\vee}\to0.
\]
The middle term of the last sequence is regarded as the space of
torus-invariant \(\QQ\)-divisors; the divisor corresponding to \(\psi\) is
identified with the element
\[
((-\psi(u_{j})),(-\psi(v_{i})))\in\QQ^{n+1}\oplus\QQ^{m+1}.
\]
Its quotient \((L_{\QQ})^{\vee}\) is then identified with the space
\(\Cl(\biProj k[\by,\bx])\otimes\QQ\) of \(\QQ\)-linear equivalence
classes of \(\QQ\)-divisors. We fix an identification \((L_{\QQ})^{\vee}=\QQ^{2}\) via the
dual basis of \((\bd,\ba)\), \((\bzero,\bc)\), and call a class
\(\bw=(w_{1},w_{2})\) \emph{ample} when a positive multiple of it is the
class of an ample Cartier divisor.
\begin{lem}
\label{lem:ample-weight}
A class \(\bw=(w_{1},w_{2})\) in \(\Cl(\biProj k[\by,\bx])\otimes\QQ\) is
ample if and only if
\[
w_{2}>0\qquad\text{and}\qquad w_{1}>w_{2}\max_{i}\frac{a_{i}}{c_{i}}.
\]
In particular \(\biProj k[\by,\bx]\) is projective, the class
\(\bw=(1+\max_{i}a_{i},1)\) being ample.
\end{lem}

\begin{proof}
Let \(\psi\) be the support function of a \(\QQ\)-divisor representing
\(\bw\) and put \(q_{j}:=-\psi(u_{j})\) and \(p_{i}:=-\psi(v_{i})\), so that
\[
w_{1}=\sum_{j}d_{j}q_{j}+\sum_{i}a_{i}p_{i},\qquad w_{2}=\sum_{i}c_{i}p_{i}
\]
by the identification just made. Then \(\bw\) is ample if and only if \(\psi\)
is strictly convex \cite[Theorem 6.1.14]{cox2011toric}.

Fix a maximal cone \(\sigma_{j,i}\) and let \(m\) be the linear function
agreeing with \(\psi\) on it, so that
\[
m(u_{j'})=-q_{j'}\ (j'\ne j),\qquad m(v_{i'})=-p_{i'}\ (i'\ne i).
\]
By Lemma \ref{lem:twisted-fan} the rays outside \(\sigma_{j,i}\) are
\(u_{j}\) and \(v_{i}\), so strict convexity asks that \(m(u_{j})>-q_{j}\)
and \(m(v_{i})>-p_{i}\).

Applying \(m\) to the second relation of \eqref{eq:fan-relations} gives
\(c_{i}m(v_{i})=\sum_{i'\ne i}c_{i'}p_{i'}=w_{2}-c_{i}p_{i}\), that is,
\[
m(v_{i})=\frac{w_{2}}{c_{i}}-p_{i},
\]
so the condition at \(v_{i}\) reads \(w_{2}>0\). Applying \(m\) to the first
relation and substituting this value,
\begin{align*}
d_{j}m(u_{j}) & =\sum_{j'\ne j}d_{j'}q_{j'}+\sum_{i'\ne i}a_{i'}p_{i'}-a_{i}m(v_{i})\\
 & =\sum_{j'\ne j}d_{j'}q_{j'}+\sum_{i'}a_{i'}p_{i'}-\frac{a_{i}}{c_{i}}w_{2}\\
 & =w_{1}-\frac{a_{i}}{c_{i}}w_{2}-d_{j}q_{j},
\end{align*}
that is, \(m(u_{j})=\bigl(w_{1}-(a_{i}/c_{i})w_{2}\bigr)/d_{j}-q_{j}\),
so the condition at \(u_{j}\) reads \(w_{1}>(a_{i}/c_{i})w_{2}\). This does
not depend on \(j\), and letting \(\sigma_{j,i}\) range over the maximal
cones, the two conditions together are the inequalities of the
statement. The last assertion holds because \(c_{i}\ge1\), so that
\(a_{i}/c_{i}\le a_{i}\).
\end{proof}

\subsection{Bigraded varieties}

A bihomogeneous ideal of \(k[\by,\bx]\) that does not contain the
irrelevant ideal \(k[\by,\bx]_{\dagger}\) defines a closed subvariety of
\(\biProj k[\by,\bx]\), and every closed subvariety is obtained in this
way \cite[Proposition 5.2.4]{cox2011toric}.  In particular, closed
subvarieties correspond to bihomogeneous prime ideals that do not contain
\(k[\by,\bx]_{\dagger}\).

\begin{defn}
\label{defn:bigraded-scheme}
A \emph{bigraded scheme} means a closed subscheme
\[
V(\bh)=V(h_{1},\dots,h_{r})
\]
of \(\biProj k[\by,\bx]\) endowed with
the datum of the defining bihomogeneous polynomials
\(\bh=(h_{1},\dots,h_{r})\) such that
\(\langle\bh\rangle\nsupseteq k[\by,\bx]_{\dagger}\). It is a
\emph{bigraded variety} when \(\langle\bh\rangle\) is prime.
\end{defn}

To a bigraded scheme \(V(\bh)\), we associate its bihomogeneous coordinate
ring \(S=k[\by,\bx]/\langle\bh\rangle\), with the induced bigrading, and
put
\[
S_{\dagger}:=k[\by,\bx]_{\dagger}\cdot S
=\langle y_{j}x_{i}\mid0\le j\le n,\ 0\le i\le m\rangle S.
\]
As above, we regard \(V(\bh)=\biProj S\) as a multigraded Proj of \(S\),
using the selected charts determined by \(S_{\dagger}\).

\subsection{Transforming bigraded varieties into monograded
ones\label{subsec:Segre-isomorphisms}}

Let \(\bw=(w_{1},w_{2})\in\ZZ^{2}\) be a primitive integral vector in
the ample cone, one being given by Lemma \ref{lem:ample-weight}. For
a bigraded ring \(S\) we call
\[
S^{(\bw)}:=\bigoplus_{t\in\NN}S_{t\bw}
\]
its \emph{\(\bw\)-diagonal}; it is a monograded ring, \(S_{t\bw}\) being
its degree-\(t\) part and \(S^{(\bw)}_{\dagger}=\bigoplus_{t>0}S_{t\bw}\)
its irrelevant ideal. In the special case \(d_{j}=c_{i}=1\) for all \(j\)
and \(i\), so that \(\deg(y_{j})=(1,0)\) and \(\deg(x_{i})=(a_{i},1)\), an
isomorphism of the form below is recorded in
\cite[\S1.4, p.~12, immediately before Lemma 1.4.2]{lavila1999ondiagonals}.
For an integral affine variety over an algebraically closed field,
\cite[Lemma 2.7 and the paragraph following it]{berchtold2006gitequivalence}
identifies the \(\Proj\) of a \(\bw\)-diagonal with a GIT quotient. The
following result treats the weights \(d_{j}\) and \(c_{i}\) occurring here,
and the chart argument in its proof needs neither hypothesis.

\begin{prop}
\label{prop:diagonal-proj}
Let \(S=k[\by,\bx]/\langle\bh\rangle\) be the bihomogeneous coordinate
ring of a bigraded scheme. There is a canonical isomorphism
\begin{equation}\label{eq:diagonal}
\Proj S^{(\bw)}\cong\biProj S.
\end{equation}
Moreover, if \(\fp\subset S\) is a bihomogeneous prime ideal with
\(\fp\nsupseteq S_{\dagger}\) and
\(\fp^{(\bw)}:=\fp\cap S^{(\bw)}\), then \(\fp^{(\bw)}\) is a homogeneous
prime ideal not containing \(S^{(\bw)}_{\dagger}\), and
\[
V(\fp)\cong\Proj\bigl(S^{(\bw)}/\fp^{(\bw)}\bigr)
\]
under \eqref{eq:diagonal}.
\end{prop}

\begin{proof}
Put \(P:=k[\by,\bx]\). For every \(j\)
and \(i\), set
\[
F_{j,i}:=
y_{j}^{\,c_{i}w_{1}-a_{i}w_{2}}x_{i}^{\,d_{j}w_{2}}.
\]
Both exponents are positive by Lemma \ref{lem:ample-weight}, and
\[
\deg(F_{j,i})
=\bigl(d_{j}c_{i}w_{1},d_{j}c_{i}w_{2}\bigr)
=d_{j}c_{i}\bw.
\]
Thus \(F_{j,i}\in P^{(\bw)}\) is homogeneous of degree \(d_{j}c_{i}\);
we use the same notation for its image in \(S^{(\bw)}\).

We claim that the \(F_{j,i}\) generate \(S^{(\bw)}_{\dagger}\) up to
radical. Let \(M:=\by^{\bbeta}\bx^{\balpha}\) be a
monomial of bidegree \(t\bw\) with \(t>0\). Its second degree shows that
some \(\alpha_{i}\) is positive. If all the \(\beta_{j}\) were zero, then
\[
\frac{w_{1}}{w_{2}}
=\frac{\sum_{i}a_{i}\alpha_{i}}{\sum_{i}c_{i}\alpha_{i}}
\le\max_{i}\frac{a_{i}}{c_{i}},
\]
contrary to Lemma \ref{lem:ample-weight}. Hence \(M\) is divisible by
some \(y_{j}x_{i}\). For a sufficiently large \(N\), we can write
\[
M^{N}=F_{j,i}M'
\]
with \(M'\) a monomial of bidegree \((Nt-d_{j}c_{i})\bw\), and hence
\(M'\in P^{(\bw)}\). The ideal \(S^{(\bw)}_{\dagger}\) is radical and
generated by these monomials \(M\), so it follows that
\[
\sqrt{\langle F_{j,i}\mid j,i\rangle S^{(\bw)}}=S^{(\bw)}_{\dagger}.
\]
Consequently, the affine opens \(D_{+}(F_{j,i})\) cover
\(\Proj S^{(\bw)}\).

Since \(F_{j,i}\) is a product of positive powers of \(y_{j}\) and
\(x_{i}\), we have a canonical graded isomorphism
\[
S_{F_{j,i}}\xrightarrow{\sim}S_{y_{j}x_{i}}.
\]
Moreover, the natural map from the diagonal localization to the
degree-zero part of the bigraded localization is an isomorphism:
\[
\bigl((S^{(\bw)})_{F_{j,i}}\bigr)_{0}
\xrightarrow{\sim}\bigl(S_{F_{j,i}}\bigr)_{\bzero}.
\]
Combining these two
isomorphisms gives
\[
\bigl((S^{(\bw)})_{F_{j,i}}\bigr)_{0}
\cong\bigl(S_{y_{j}x_{i}}\bigr)_{\bzero}.
\]
These identifications give isomorphisms
\(D_{+}(F_{j,i})\cong D_{+}(y_{j}x_{i})\) that commute with
further localization. They therefore glue to \eqref{eq:diagonal}.

Finally, let \(\fp\nsupseteq S_{\dagger}\). Then some \(y_{j}x_{i}\)
does not belong to \(\fp\), and primeness gives \(F_{j,i}\notin\fp\).
Since \(F_{j,i}\) lies in \(S^{(\bw)}_{\dagger}\), the ideal \(\fp^{(\bw)}\)
does not contain \(S^{(\bw)}_{\dagger}\). It is prime, being the
contraction of \(\fp\) along the inclusion \(S^{(\bw)}\subset S\), and it
is homogeneous because \(\fp\) is bihomogeneous and the degree-\(t\) part
of \(S^{(\bw)}\) is \(S_{t\bw}\). The natural graded isomorphism
\[
S^{(\bw)}/\fp^{(\bw)}\cong(S/\fp)^{(\bw)}
\]
and the first assertion applied to \(S/\fp\) give the last assertion.
\end{proof}

For \(\ba=\bzero\) and \(\bw=(1,1)\), the \(\bw\)-diagonal of \(k[\by,\bx]\)
is the Segre product
\[
k[\by]\sharp k[\bx]=\bigoplus_{t\in\NN}k[\by]_{t}\otimes k[\bx]_{t},
\]
\(\biProj k[\by,\bx]\) is the product \(\PP(\bd)\times\PP(\bc)\), and
\eqref{eq:diagonal} specializes to the classical isomorphism between
\(\Proj\) of a Segre product and the product of the two \(\Proj\)'s
\cite[Chapter II, Exercise 5.11]{hartshorne1977algebraic}.

To compute an explicit presentation of \(T:=k[\by,\bx]^{(\bw)}\),
consider the monoid
\[
A_{\bw}:=\left\{(\bbeta,\balpha)\in\NN^{n+1}\times\NN^{m+1}\bigmid w_{2}\Bigl(\sum^{n}_{j=0}d_{j}\beta_{j}+\sum^{m}_{i=0}a_{i}\alpha_{i}\Bigr)=w_{1}\sum^{m}_{i=0}c_{i}\alpha_{i}\right\}
\]
of the exponents of those monomials \(\by^{\bbeta}\bx^{\balpha}\) whose
bidegree lies on the ray spanned by \(\bw\); since \(\gcd(w_{1},w_{2})=1\),
that bidegree is then \(t\bw\) for some \(t\in\NN\). This monoid consists
of the lattice points in the rational polyhedral cone in \(\RR^{n+m+2}\)
defined as the intersection of \((\RR_{\ge0})^{n+m+2}\) and the
hyperplane displayed above. As is well known, we can compute the
Hilbert basis of this monoid, that is, the (unique) minimal generating
set.
\begin{lem}
\label{lem:diagonal-generators}
Let \((\bbeta_{1},\balpha_{1}),\dots,(\bbeta_{s},\balpha_{s})\in\NN^{n+1}\times\NN^{m+1}\)
be the Hilbert basis of the monoid \(A_{\bw}\). Then the \(\bw\)-diagonal
\(T=k[\by,\bx]^{(\bw)}\) is the \(k\)-subalgebra of \(k[\by,\bx]\) generated
by the monomials
\[
\by^{\bbeta_{1}}\bx^{\balpha_{1}},\dots,\by^{\bbeta_{s}}\bx^{\balpha_{s}}.
\]
\end{lem}

\begin{proof}
It is easy to see that \(T\) is generated by the monomials
\(\by^{\bbeta}\bx^{\balpha}\), \((\bbeta,\balpha)\in A_{\bw}\), as a
\(k\)-module. Every element \((\bbeta,\balpha)\in A_{\bw}\) is written as
\[
(\bbeta,\balpha)=\sum^{s}_{i=1}\nu_{i}(\bbeta_{i},\balpha_{i})\quad(\nu_{i}\in\NN).
\]
Thus, we can write
\(\by^{\bbeta}\bx^{\balpha}=\prod_{i}(\by^{\bbeta_{i}}\bx^{\balpha_{i}})^{\nu_{i}}\),
which shows the lemma.
\end{proof}

The lemma enables us to explicitly compute \(T\) as the image of the map
\begin{equation}\label{eq:Segre}
k[Z_{1},\dots,Z_{s}]\to k[\by,\bx],\quad Z_{i}\mapsto\by^{\bbeta_{i}}\bx^{\balpha_{i}}
\end{equation}
and hence as a quotient of the polynomial ring \(k[Z_{1},\dots,Z_{s}]\)
by the kernel.

Now let \(V(\bh)\subset\biProj k[\by,\bx]\) be a bigraded scheme with
bihomogeneous coordinate ring \(S=k[\by,\bx]/\langle\bh\rangle\). Each
\(S_{t\bw}\) is the image of \(k[\by,\bx]_{t\bw}\), so \(S^{(\bw)}\) is the
image of \(T\), and map \eqref{eq:Segre} gives an explicit description of
\(S^{(\bw)}\) as well as of the map \(S^{(\bw)}\to S\). Consider a closed
subvariety \(Z=V(\fp)\subset V(\bh)\), where \(\fp\) is a bihomogeneous
prime ideal of \(S\) with \(\fp\nsupseteq S_{\dagger}\). Then the
monograded representation of \(Z\) is given by
\(\fp|_{S^{(\bw)}}:=\fp\cap S^{(\bw)}\) by Proposition
\ref{prop:diagonal-proj}. We can compute \(\fp|_{S^{(\bw)}}\) from \(\fp\).

In particular, for monograded varieties \(Y=\Proj S\) with
\(S=k[\by]/\langle\bg\rangle\) and \(X=\Proj R\) with
\(R=k[\bx]/\langle\bf\rangle\), the case \(\ba=\bzero\), \(\bw=(1,1)\) gives
the product
\[
Y\times X=\biProj\bigl(S\otimes R\bigr)
=\biProj\bigl(k[\by,\bx]/\langle\bf,\bg\rangle\bigr)
\subset\PP(\bd)\times\PP(\bc)
\]
together with its monograded representation \(\Proj S\sharp R\) by the
Segre product \(S\sharp R=(S\otimes R)^{((1,1))}\). This is the case used
in Section \ref{subsec:bi-to-mono-projections}.
\subsection{Bi-to-mono projections and graph morphisms\label{subsec:bi-to-mono-projections}}

Let \(W\) be a bigraded variety and let \(T\) be its bihomogeneous
coordinate ring. The elements of bidegree \((s,0)\) in \(k[\by,\bx]\) are
exactly those of \(k[\by]\), so
\[
R:=\bigoplus_{s\ge0}T_{s,0}
\]
is the image of \(k[\by]\) in \(T\), a monograded ring. The projection
\(\biProj k[\by,\bx]\to\PP(\bd)\) of Section
\ref{subsec:monograded-bigraded} restricts on \(W\) to a morphism
\[
W\to V:=\Proj R,
\]
whose image is the closed subvariety \(V\) of \(\PP(\bd)\): that
projection is proper, so the image is closed, and \(\Proj R\) is the
scheme-theoretic image of \(W\).
If \(\fp\subset k[\by,\bx]\) is the bihomogeneous prime ideal defining
\(W\), then \(R=k[\by]/\fp|_{k[\by]}\), and we can compute a generating set
of \(\fp|_{k[\by]}\) from one of \(\fp\).
\begin{defn}
\label{defn:bi-to-mono}
We call the morphism \(W\to V\) as above from the bigraded variety
\(W\) to the monograded variety \(V\) a \emph{bi-to-mono projection},
assuming that it is endowed with the data of a finite generating set of
\(\fp|_{k[\by]}\).
\end{defn}
For \(\ba=\bzero\) the ambient \(\biProj k[\by,\bx]\) is the product
\(\PP(\bd)\times\PP(\bc)\), the projection above is \(\pr_{1}\), and by
symmetry \(\pr_{2}\) gives a bi-to-mono projection as well, after
exchanging the two blocks of variables. For \(\ba\ne\bzero\) there is
no morphism to \(\PP(\bc)\), and only \(\pr_{1}\) is available.

Giving a morphism \(f\colon Y\to X\) of monograded varieties is
equivalent to giving its graph \(\Gamma_{f}\), which is, by definition, the closed subvariety \(\{(y,f(y))\mid y\in Y\}\subset Y\times X\).
\begin{defn}
For monograded varieties \(Y=\Proj S\) and \(X=\Proj R\), a \emph{(bigraded)
graph morphism} \(f\colon Y\to X\) is a morphism given with an explicit
representation of its graph \(\Gamma_{f}\subset Y\times X=\biProj(S\otimes R)\)
as a bigraded variety.
\end{defn}

A graph morphism \(f\colon Y\to X\) induces two bi-to-mono projections
\[
Y\xleftarrow{\sim}\Gamma_{f}\to X,
\]
with the first one being an isomorphism. Conversely, if \(Y\) and \(X\)
are monograded varieties and if \(G\subset Y\times X\) is a bigraded
variety such that the projection \(G\to Y\) is an isomorphism, then
\(G\) defines a graph morphism \(f\colon Y\to X\) such that \(\Gamma_{f}=G\).
The graph morphism is our primary choice of data type to represent
morphisms of varieties.

At some steps of our algorithm for the minimal model program, we need
to transform a bi-to-mono projection into a graph morphism,
transforming the source bigraded variety into a monograded variety. We
need this operation in order to avoid dealing with ``trigraded
varieties,'' ``quadgraded varieties,'' or such.
\begin{lem}
\label{lem:bi-to-mono-to-graph}
Let \(f\colon W\to V\) be a bi-to-mono projection and let \(W^{\rm}\)
be the monograded variety that corresponds to \(W\) by Proposition
\ref{prop:diagonal-proj}. Then we can compute the graph
\(\Gamma_{f}\subset W^{\rm}\times V\) of \(f\) as a bigraded variety.
\end{lem}

\begin{proof}
Put \(P:=k[\by,\bx]\) and let \(\fp\subset P\) be the bihomogeneous prime
ideal that defines \(W\). Write
\[
\iota\colon\biProj P\xrightarrow{\ \sim\ }\Proj P^{(\bw)}
\]
for the inverse of the isomorphism \eqref{eq:diagonal} of Proposition
\ref{prop:diagonal-proj}, and \(p\colon\biProj P\to\PP(\bd)=\Proj k[\by]\)
for the projection. The bigraded product
\(\Proj P^{(\bw)}\times\PP(\bd)\) is written as
\(\biProj\bigl(P^{(\bw)}\otimes k[\by]\bigr)\), and it contains \(W^{\rm}\times V\).
Consider the morphism
\[
(\iota,p)\colon\biProj P\to\Proj P^{(\bw)}\times\PP(\bd),\quad z\mapsto(\iota(z),p(z)),
\]
corresponding to the ring homomorphism
\[
\xi\colon P^{(\bw)}\otimes k[\by]\to P,\quad u\otimes r\mapsto ur,
\]
both \(P^{(\bw)}\) and \(k[\by]\) being subrings of \(P\). The map \(\xi\) is
homogeneous with respect to the map
\[
q\colon\ZZ^{2}\to\ZZ^{2},\quad(t,n)\mapsto t\bw+(n,0),
\]
which is injective because \(w_{2}>0\). The desired graph
\(\Gamma_{f}\subset W^{\rm}\times V\) is the image of \(W\) by
\((\iota,p)\). We claim that \(\xi^{-1}(\fp)\) defines \(\Gamma_{f}\) as a
subvariety of the bigraded scheme \(\Proj P^{(\bw)}\times\PP(\bd)\). To
see this, write an element \(h\in P^{(\bw)}\otimes k[\by]\) as
\(h=\sum_{\bv\in\ZZ^{2}}h_{\bv}\), where \(h_{\bv}\) is the degree-\(\bv\)
part of \(h\). We have \(\xi(h)=\sum_{\bv\in\ZZ^{2}}\xi(h_{\bv})\). Since
\(q\) is injective, \(\xi(h_{\bv})\) is the degree \(q(\bv)\)-part of
\(\xi(h)\). Since \(\fp\) is bihomogeneous, we have that \(\xi(h)\in\fp\)
if and only if \(\xi(h_{\bv})\in\fp\) for every \(\bv\). This implies
that \(\xi^{-1}(\fp)\) is bihomogeneous. The ideal \(\xi^{-1}(\fp)\) is
generated by bihomogeneous elements that vanish along
\((\iota,p)(W)\), which shows the claim. With \(\fp^{(\bw)}\) as in
Proposition \ref{prop:diagonal-proj}, we have
\(W^{\rm}\times V=\biProj\bigl((P^{(\bw)}/\fp^{(\bw)})\otimes(k[\by]/\fp|_{k[\by]})\bigr)\),
and \(\Gamma_{f}\) is defined there by the image of \(\xi^{-1}(\fp)\). We
can compute this ideal as is well known.
\end{proof}

Lemma \ref{lem:bi-to-mono-to-graph} shows how to transform a
bi-to-mono projection to a graph morphism of monograded varieties
without changing the target monograded variety.

\subsection{Homogeneous morphisms\label{subsec:Homogeneous-morphisms}}

Let \(Y=\Proj S\) and \(X=\Proj R\) be monograded varieties. We call
a morphism \(\Proj S\to\Proj R\) induced by a homogeneous (that is,
degree-preserving) homomorphism \(\phi\colon R\to S\) a \emph{homogeneous
morphism}. Note that \(\phi\) must satisfy \(\sqrt{\langle\phi(R_{\dagger})\rangle}=S_{\dagger}\)
so that the domain of the morphism is the entire \(\Proj S\). Note
that not every morphism of monograded varieties is expressed as a
homogeneous morphism. However, homogeneous morphisms are useful in
some constructions.

Given a homogeneous morphism \(f\colon Y\to X\), we can effectively
compute its graph \(\Gamma_{f}\) in the bigraded scheme \(\biProj\bigl(S\otimes R\bigr)\).
Indeed, its bihomogeneous defining ideal is the kernel of
\[
S\otimes R\to S,\quad s\otimes r\mapsto s\cdot\phi(r),
\]
which can be computed. Thus, we can transform a homogeneous morphism
into a graph morphism.

\section{Nef thresholds and nefness of canonical divisors
\label{sec:nefness-canonical}}

Let \(X=\Proj R\) be a normal monograded variety with \(R=k[\bx]/\fp\). In this section, we compute the nef threshold
\[
\lambda:=\inf\{t\in\RR_{\ge0}\mid K_X+tH\text{ is nef}\}
\]
for a given ample Cartier divisor \(H\) 
and decide whether the canonical divisor is nef.
From \cite[(2.2.9)]{goto1978ongraded} and
\cite[14.5.2]{brodmann2012localcohomology}, the canonical module \(\omega_{R}\) of
\(R\) is isomorphic to the Ext group
\(\Ext^{t}_{k[\bx]}(R,\omega_{k[\bx]})\) with
\(t=\dim k[\bx]-\dim R\). The coherent sheaf \(\widetilde{\omega_{R}}\)
associated to the graded \(R\)-module \(\omega_{R}\) is the canonical
sheaf \(\omega_{X}\) of \(X\) \cite[5.1.8]{goto1978ongraded}. From
the obtained representation of \(\omega_{R}\), we can find a canonical
divisor \(K_{X}\) (that is, a Weil divisor such that \(\cO_{X}(K_{X})\cong\omega_{X}\))
in a way explained at \cite[Section 3]{schwede2018divisor}. More
precisely, we can specify a Weil divisor \(K_{X}=\sum^{l}_{i=1}n_{i}D_{i}\)
by giving integers \(n_{i}\) and generators of homogeneous prime ideals
\(\fp_{i}\subset R\) defining \(D_{i}\).

\begin{prop}
\label{prop:nef-scaled}
Let \(X\) be a normal projective variety
of dimension \(d\) with only log terminal singularities, let \(H\) be an ample
Cartier divisor on \(X\) and let \(t\) be a positive rational number. Let \(N\)
be a positive integer such that \(N(K_{X}+tH)\) is Cartier and put
\[
m:=\begin{cases}
7 & (d=3\text{ and }N=1),\\
6 & (d=3\text{ and }N\ge2),\\
2^{d+1}(d+1)!(d+1) & (d\ne3).
\end{cases}
\]
Then \(K_{X}+tH\) is nef if and only if the linear system \(|mN(K_{X}+tH)|\)
is base-point-free.
\end{prop}

\begin{proof}
If \(|mN(K_{X}+tH)|\) is base-point-free, then \(mN(K_{X}+tH)\) is nef, and
hence so is \(K_{X}+tH\).
Conversely, suppose that \(D:=K_{X}+tH\) is nef and put \(L:=ND\), which is a
nef Cartier divisor. For a positive rational number \(a\) we have
\[
aL-K_{X}=aND-K_{X}=(aN-1)D+tH;
\]
in particular \((1/N)L-K_{X}=tH\) is ample.

Suppose first that \(d=3\). The ample cone being open, we may choose a
rational number \(\delta\) with \(0<\delta<1\) such that \(tH-\delta D\) is
ample. Putting \(a:=(1-\delta)/N\), the divisor
\[
aL-K_{X}=tH-\delta D
\]
is then ample and in particular nef and big. By
\cite[Theorem 4]{matsushita1996effective}, applied with the zero boundary,
the system \(|nL|\) is base-point-free for every integer \(n>a+11/2\). If
\(N=1\), then \(a<1\) and we may take \(n=7\). If \(N\ge2\), then
\(a<1/N\le1/2\) and we may take \(n=6\).

Suppose next that \(d\ne3\). Putting \(a:=1/N\), the divisor \(aL-K_{X}=tH\) is
ample and \(\lceil a\rceil=1\) because \(0<a\le1\). By the effective base
point free theorem
\cite[Theorem 1.1 and Remark 1.2]{fujino2009effective}, the system
\(|mL|\) is base-point-free.
\end{proof}

\begin{rem}
\label{rem:fujino-bound}
For \(d=3\) the bound of the third case reads \(m=1536\), and applying
Fujino's theorem to \(a:=2/N\), as one also may, gives \(1920\) when \(N=1\).
It is to avoid constants of this size that dimension three is treated
separately.
\end{rem}

\begin{lem}
\label{lem:lambda}
Let \(X\) be a normal projective variety of
dimension \(d\) with only log terminal singularities such that \(K_{X}\) is
not nef, and let \(H\) be an ample Cartier divisor on \(X\). Put
\[
\lambda:=\inf\{t\in\RR_{\ge0}\mid K_{X}+tH\text{ is nef}\}.
\]
Then the following hold.
\begin{enumerate}
\item \(\{t\in\RR_{\ge0}\mid K_{X}+tH\text{ is nef}\}=[\lambda,\infty)\) and
\(0<\lambda<\infty\).
\item \(K_{X}+\lambda H\) is nef but not ample.
\item \(\lambda\) is a rational number. Moreover, if \(a\in\ZZ_{>0}\) is such
that \(aK_{X}\) is Cartier and we write \(\lambda\) in lowest terms, then its
numerator is at most \(a(d+1)\).
\end{enumerate}
\end{lem}

\begin{proof}
Assertions (1) and (2) follow from the following facts:
\begin{itemize}
\item the nef cone is a closed convex cone and the ample cone is its
interior,
\item the class \([K_{X}]\) is outside of the nef cone and \([H]\) is in the
ample cone.
\end{itemize}
The half-line \(\{[K_{X}]+t[H]\mid t\in\RR_{\ge0}\}\) intersects the nef cone in the closed interval \([\lambda,\infty)\) with \(\lambda>0\), and the open interval \((\lambda,\infty)\) is contained in the ample cone.

For (3), we may assume that \(a\) is the least positive integer with \(aK_{X}\)
Cartier, since the asserted bound only increases with \(a\). Write \(a\lambda\)
and \(\lambda\) in lowest terms as \(v/u\) and \(v'/u'\), respectively. We have
\(v\ge v'\). By the rationality theorem
\cite[Theorem 4-1-1]{kawamata1987introduction}, we have \(v\le a(d+1)\) and
hence \(v'\le a(d+1)\).
\end{proof}

\begin{lem}[{cf.~\cite[Propositions 10.1 and 10.3]{yasuda2023theisomorphism}}]
\label{lem:basept-free}Suppose that we are given a monograded variety
\(X=\Proj R\) and an invertible sheaf \(\cL\) on \(X\) together with
a datum of a graded \(R\)-module inducing it.
\begin{enumerate}
\item We can algorithmically determine whether \(\cL\) is
base-point-free (equivalently, globally generated).
\item If \(\cL\) is base-point-free, then we can algorithmically compute the associated morphism
\(\Phi_{|\cL|}\colon X\to\PP^{n-1}\) with \(n=h^{0}(X,\cL)\) as a graph
morphism.
\end{enumerate}
\end{lem}

\begin{proof}
The first assertion is the same as \cite[Proposition 10.1]{yasuda2023theisomorphism}
except that \(R\) is not necessarily standard graded (that is, generated
in degree one) in our present setting. However, the proof of the cited
result is valid also in our setting. For the second assertion, the
proof of \cite[Proposition 10.3]{yasuda2023theisomorphism} explains
how to compute \(\Phi_{|\cL|}\) as a graph morphism in the standard
graded setting. This also applies to our present setting.
\end{proof}

\begin{prop}
\label{prop:lambda-computable}
Suppose that we are given a normal \(\QQ\)-Gorenstein monograded variety
\(X\) of dimension \(d\) with only log terminal singularities such that
\(K_{X}\) is not nef, an ample Cartier divisor \(H\) on \(X\), and a positive
integer \(a\) such that \(aK_{X}\) is Cartier. Then we can algorithmically
compute the rational number
\[
\lambda=\inf\{t\in\RR_{\ge0}\mid K_{X}+tH\text{ is nef}\}
\]
of Lemma \ref{lem:lambda}. See Algorithm \ref{algo:lambda}.
\end{prop}

\begin{proof}
We first note that for a positive rational number \(t=p/q\), written with
positive integers \(p\) and \(q\), we can decide whether \(K_{X}+tH\) is nef.
Indeed, \(N:=aq\) is a positive integer such that
\[
N(K_{X}+tH)=q\cdot aK_{X}+ap H
\]
is Cartier, and hence, with \(m\) as in Proposition \ref{prop:nef-scaled},
the divisor \(K_{X}+tH\) is nef if and only if the Cartier divisor
\(mN(K_{X}+tH)\) is base-point-free, which we can decide by Lemma
\ref{lem:basept-free}.

Testing \(t=1,2,4,\dots\) in this way, we find the first \(\beta=2^{i}\) with
\(K_{X}+\beta H\) nef. Testing then \(t=\beta/2,\beta/4,\dots\), we find the first
\(\alpha=\beta/2^{j}\) with \(K_{X}+\alpha H\) not nef. By Lemma \ref{lem:lambda}(1) we
have \(\alpha<\lambda\le\beta\).

By Lemma \ref{lem:lambda}(3), writing \(\lambda=v/u\) in lowest terms, we
have \(v\le a(d+1)\), and \(\alpha<v/u\le\beta\) gives \(v/\beta\le u<v/\alpha\).
Hence the set \(C\) of rational numbers \(v/u\) in lowest terms with
\(1\le v\le a(d+1)\) and \(\alpha<v/u\le\beta\) is finite, it contains
\(\lambda\), and we can list its elements in increasing order. Testing them
in this order, the first one at which \(K_{X}+tH\) is nef is \(\lambda\).
\end{proof}

\begin{algorithm}[ht]
\raggedright
\caption{Computing the threshold \(\lambda\)}
\label{algo:lambda}

\textbf{Input:} A normal \(\QQ\)-Gorenstein log terminal monograded variety
\(X\) of dimension \(d\) such that \(K_{X}\) is not nef, an ample Cartier
divisor \(H\) on \(X\), and a positive integer \(a\) such that \(aK_{X}\) is
Cartier.

\textbf{Output:} The rational number
\(\lambda=\inf\{t\in\RR_{\ge0}\mid K_{X}+tH\text{ is nef}\}\).

\textbf{Procedure:} Here we decide the nefness of \(K_{X}+tH\) for a
positive rational number \(t=p/q\) by testing whether \(mN(K_{X}+tH)\) is
base-point-free, where \(N=aq\) and \(m\) is as in Proposition
\ref{prop:nef-scaled}.
\begin{enumerate}
\item Let \(\beta\) be the first of \(1,2,4,\dots\) with \(K_{X}+\beta H\) nef.
\item Let \(\alpha\) be the first of \(\beta/2,\beta/4,\dots\) with
\(K_{X}+\alpha H\) not nef.
\item List the finitely many rational numbers \(v/u\) in lowest terms with
\(1\le v\le a(d+1)\) and \(\alpha<v/u\le\beta\) in increasing order.
\item Return the first one of them, say \(t\), with \(K_{X}+tH\) nef.
\end{enumerate}
\end{algorithm}

\begin{lem}
\label{lem:ample-cartier}
Let \(X=V(\bf)\subset\PP(\bc)\) be a monograded variety with homogeneous
coordinate ring \(R=k[\bx]/\langle\bf\rangle\) and put
\(l:=\lcm(c_{0},\dots,c_{m})\). Then the sheaf \(\cO_{X}(l)\) is invertible
and ample, and we can compute an effective Cartier divisor \(H\) on \(X\)
with \(\cO_{X}(H)\cong\cO_{X}(l)\).
\end{lem}

\begin{proof}
Since \(l\) is divisible by every \(c_i\), the sheaf
\(\cO_{\PP(\bc)}(l)\) is invertible: it is trivialized on \(D_+(x_i)\) by
\(x_i^{l/c_i}\). It is ample because a sufficiently high power
\(\cO_{\PP(\bc)}(sl)\) is very ample, by the usual Veronese description of
\(\Proj\) \cite[Tags~0EGH and 0B5J]{stacksprojectauthors2022stacksproject}.
Hence \(\cO_X(l)\) is invertible and ample. Choose \(i\) with
\(x_i\notin\langle\bf\rangle\). Then \(x_i^{l/c_i}\) is a nonzero section of
\(\cO_X(l)\), and its zero divisor is a computable effective Cartier divisor
\(H\) satisfying \(\cO_X(H)\cong\cO_X(l)\).
\end{proof}

\begin{rem}
\label{rem:choice-ample-divisor}
For the ample divisor \(H\) of Lemma \ref{lem:ample-cartier} and for
every Weil divisor \(E\), we have \(\cO_X(E+jH)\cong\cO_X(E)(jl)\).
Thus, twisting by a multiple of \(H\) amounts to just shifting the
grading of the corresponding graded module. So the ample Cartier
divisor \(H\) need not be given as part of the input in Proposition
\ref{prop:lambda-computable}: we may always use the divisor of Lemma
\ref{lem:ample-cartier}. For the same reason, Algorithm
\ref{algo:MMP} chooses such a new ample divisor after each step
rather than carrying the previous one forward, since its birational
transform need not remain ample or correspond to a twist of the
grading on the new model.
\end{rem}

It remains to treat \(t=0\), which the rest of the section has avoided.

\begin{prop}
\label{prop:nef}
Assume that we are given a monograded variety \(X\) of dimension three
with only log terminal singularities. Then we can algorithmically check
whether \(K_{X}\) is nef.
\end{prop}

\begin{proof}
Let \(H\) be an ample Cartier divisor on \(X\), produced by Lemma
\ref{lem:ample-cartier}. Since \(X\) has only log terminal singularities,
\(K_X\) is \(\QQ\)-Cartier. For \(r=1,2,\dots\), compute
\(\omega_X^{[r]}\) and test whether it is invertible, using Fitting ideals.
This search terminates, and for the first successful \(r\), the divisor
\(rK_X\) is Cartier. We run two searches at once. On one side we test
whether \(\omega^{[ir]}_{X}\) is globally generated, for \(i=1,2,\dots\) in
turn. On the other we test whether \(K_{X}+2^{-j}H\) is nef, for
\(j=1,2,\dots\) in turn; by Proposition \ref{prop:nef-scaled} each of
these is a single test of base point freeness.

One of the two stops.
If \(K_{X}\) is nef then it is semi-ample, the
abundance conjecture being a theorem in dimension three
\cite[Cors.~4.7.8 and~4.7.9]{fujino2017foundations}, so
\(\omega^{[ir]}_{X}\) is globally generated for some \(i\).
If \(K_{X}\) is
not nef then \(\lambda>0\) by Lemma \ref{lem:lambda}, so \(K_{X}+2^{-j}H\)
fails to be nef once \(2^{-j}<\lambda\).

Either outcome answers the question. Global generation of
\(\omega^{[ir]}_{X}\) makes \(K_{X}\) nef. And if \(K_{X}+2^{-j}H\) is not
nef for some \(j\), then \(K_{X}\) is not nef, since the sum of a nef and
an ample divisor is ample.
\end{proof}

\section{Bigraded global Hom modules\label{sec:Bigraded-global-Hom}}

In this section, we prove a bigraded version of Smith's result \cite[Theorem 1]{smith2000computing}
on computing global Hom modules.
Let
\begin{align*}
S & =k[x_{1,0},\dots,x_{1,d_{1}},x_{2,0},\dots,x_{2,d_{2}}]
\end{align*}
be a polynomial ring with \(d_{1}+d_{2}+2\) variables. We give a \(\ZZ^{2}\)-grading
to \(S\) by
\begin{align*}
\deg x_{1,t} & :=(c_{1,t},0)\quad(0\le t\le d_{1},\,c_{1,t}\in\ZZ_{>0}),\\
\deg x_{2,t} & :=(0,c_{2,t})\quad(0\le t\le d_{2},\,c_{2,t}\in\ZZ_{>0}).
\end{align*}
We write \(\bd:=(d_{1},d_{2})\) and \(|\bd|:=d_{1}+d_{2}\), and
\(\bc_{s}:=(c_{s,0},\dots,c_{s,d_{s}})\) for the weights of the \(s\)-th
block, so that \(\biProj S=\PP(\bc_{1})\times\PP(\bc_{2})\).
Let \(I\subset S\) be a bihomogeneous (not necessarily prime) ideal that does not contain the irrelevant ideal \(S_{\dagger}\), and consider the bigraded scheme \(X := \biProj R\) with \(R := S/I\).  Let
\[
\mathcal X:=\left[(\Spec R\setminus V(R_\dagger))/\GG_m^2\right]
\]
be the stacky counterpart of $X$, which has $X$ as its coarse moduli space.
The algorithm presented in this section computes global Hom modules on the Deligne--Mumford stack \(\mathcal X\).

Let \(M\) be a finitely generated bigraded \(R\)-module and let
\[
F_{\bullet}(M)\colon\cdots\to F_{1}(M)\to F_{0}(M)\to0
\]
be the unique minimal free bigraded resolution of \(M\). We write
\[
F_{i}(M)=\bigoplus^{b_{i}(M)}_{j=1}R(-\ba_{i,j}(M))\quad(\ba_{i,j}(M)=(a_{1,i,j}(M),a_{2,i,j}(M))\in\ZZ^{2}).
\]

\begin{defn}
For each
\(s\in\{1,2\}\) and \(i\ge0\), we define
\begin{align*}
\overline{a}_{s,i}(M)
&:=\sup\{a_{s,i,j}(M)\mid1\le j\le b_i(M)\} \in \ZZ\cup\{-\infty\},\\
\underline{a}_{s,i}(M)
&:=\inf\{a_{s,i,j}(M)\mid1\le j\le b_i(M)\} \in \ZZ\cup\{+\infty\}.
\end{align*}
Thus \(\overline{a}_{s,i}(M)=-\infty\) and
\(\underline{a}_{s,i}(M)=+\infty\) when \(b_i(M)=0\).
For \(\bi=(i_{1},i_{2})\in\NN^{2}\), we define
\begin{align*}
\overline{\ba}_{\bi}(M) & :=(\overline{a}_{1,i_{1}}(M),\overline{a}_{2,i_{2}}(M)),\\
\underline{\ba}_{\bi}(M) & :=(\underline{a}_{1,i_{1}}(M),\underline{a}_{2,i_{2}}(M)).
\end{align*}
For \(i\in\NN\), we define \(\overline{\ba}_{i}(M):=\overline{\ba}_{(i,i)}(M)\)
and \(\underline{\ba}_{i}(M)=\underline{\ba}_{(i,i)}(M)\). Let us denote
by \(_{S}M\) the same module \(M\) regarded as a bigraded \(S\)-module
via the quotient map \(S\to R\). Similarly, we define the invariants
\[
\overline{a}_{s,i}({}_{S}M),\,\underline{a}_{s,i}({}_{S}M),\,\overline{\ba}_{\bi}({}_{S}M),\,\underline{\ba}_{\bi}({}_{S}M)
\]
by using the minimal free graded resolution of \(M\) as an \(S\)-module.
\end{defn}

Let \(M\) and \(N\) be two finitely generated bigraded \(R\)-modules. We can associate  coherent
\(\cO_{\mathcal X}\)-modules
 \(\cM\) and \(\cN\) to them, respectively.   Indeed, the bigradings define
\(\GG_m^2\)-equivariant coherent sheaves on \(\Spec R\), whose
restrictions to \(\Spec R\setminus V(R_\dagger)\) descend to \(\mathcal X\).
Let
\[
\cF_{\bullet}(M)\colon\cdots\to\cF_{1}(M)\to\cF_{0}(M)\to0
\]
be the locally free resolution of \(\cM\) that is induced by \(F_{\bullet}(M)\).
For each \(\bv\in\ZZ^{2}\), we denote by \(\cN(\bv)\) the twist of \(\cN\)
by \(\bv\); it is the sheaf associated to the shifted bigraded module
\(N(\bv)\), where \(N(\bv)_{\bu}=N_{\bu+\bv}\).
We have the exact sequences
\begin{equation}\label{eq:Hom-seq}
0\to\Hom_{R}(M,N)\to\Hom_{R}(F_{0}(M),N)\to\Hom_{R}(F_{1}(M),N)
\end{equation}
and
\begin{equation}\label{eq:Hom-seq2}
0\to\Hom_{\mathcal X}(\cM,\cN(\bv))\to\Hom_{\mathcal X}(\cF_{0}(M),\cN(\bv))\to\Hom_{\mathcal X}(\cF_{1}(M),\cN(\bv)).
\end{equation}

For \(\bi=(i_{1},i_{2})\) and \(\bj=(j_{1},j_{2})\in\ZZ^{2}\), we write
\(\bi\ge\bj\) if \(i_{1}\ge j_{1}\) and \(i_{2}\ge j_{2}\), and we write
\(\bmax\) and \(\bmin\) for the componentwise maximum and minimum. For \(\be\in\ZZ^{2}\),
we define the homogeneous submodule \(M_{\ge\be}\subset M\) by
\[
M_{\ge\be}:=\bigoplus_{\bv\ge\be}M_{\bv}.
\]

\begin{prop}[{cf.~\cite[Proposition 2.5]{smith2000computing}}]
\label{prop:Hom}Let \(\bone:=(1,1)\) and
\[
\bc:=(c_{1},c_{2})\quad\left(c_{i}:=\sum^{d_{i}}_{t=0}c_{i,t}\right).
\]
Suppose that \(M\ne0\), and put
\begin{multline*}
\be_{0}:=\bmax\{\overline{\ba}_{\bd+\bone}(_{S}N),\overline{\ba}_{\bd}(_{S}N),\overline{\ba}_{|\bd|+1}(_{S}N),\overline{\ba}_{|\bd|}(_{S}N)\}-\underline{\ba}_{0}(M)-\bc+\bone\\
\in(\ZZ\cup\{-\infty\})^{2}.
\end{multline*}
Then for every \(\be\in\ZZ^2\) such that \(\be\ge\be_{0}\), there is
a natural isomorphism of bigraded \(R\)-modules
\[
\bigoplus_{\bv\ge\be}\Hom_{\mathcal X}(\cM,\cN(\bv))
\cong\Hom_{R}(M,N)_{\ge\be}.
\]
\end{prop}

\begin{proof}
We try to mimic Smith's proof in the monograded setting. However,
there is a difficulty coming from the fact that the irrelevant ideal
\(S_{\dagger}\), denoted by \(B\) in this proof, is not {*}local. We
then use a Mayer--Vietoris sequence to reduce the problem to \(*\)local
situations. Since the proof is long, we divide it into four steps.

\textbf{Step 1:} We first reduce the problem to showing that local cohomology
groups of \(N\) vanish in some degrees.
The two sides of the isomorphism
of the proposition fit into exact sequences (\ref{eq:Hom-seq2}) and
(\ref{eq:Hom-seq}), respectively. Therefore, it is enough to show
that the natural map
\[
\Hom_{R}(F_{i}(M),N)_{\bv}\to\Hom_{\mathcal X}(\cF_{i}(M),\cN(\bv))
\]
is an isomorphism for \(\bv\ge\be\) and \(i\in\{0,1\}\).
For a bigraded \(R\)-module \(Q\), let \(\widetilde Q\) denote the associated
\(\GG_m^2\)-equivariant sheaf on
\(U_R:=\Spec R\setminus V(R_\dagger)\), and
let \(\cQ\) denote its descent to
\(\mathcal X=[U_R/\GG_m^2]\).  The character
decomposition of the equivariant sections gives a natural identification
\[
\bigoplus_{\bv\in\ZZ^2}H^0(\mathcal X,\cQ(\bv))
\cong \Gamma(U_R,\widetilde Q),
\]
where the \(\bv\)-summand on the left corresponds to the
\(\bv\)-character subspace on the right.  Applying the multigraded
Serre--Grothendieck correspondence
\cite[(6.3.1)]{maclagan2004multigraded} to
\(Q={}_{S}\Hom_R(F_i(M),N)\), and using the preceding identification,
gives the following exact sequence of bigraded \(S\)-modules.
\begin{multline}\label{eq:exseq}
0\to H^{0}_{B}(_{S}\Hom_{R}(F_{i}(M),N))\to\Hom_{R}(F_{i}(M),N)\to\\
\to\bigoplus_{\bv\in\ZZ^{2}}\Hom_{\mathcal X}(\cF_{i}(M),\cN(\bv))\to H^{1}_{B}(_{S}\Hom_{R}(F_{i}(M),N))\to0.
\end{multline}
We also used
the identifications \(\Hom_{R}(F_{i}(M),N)=\Hom_{S}(_{S}F_{i}(M),_{S}N)\)
and the identification
\[
\Hom_{\mathcal X}(\cF_i(M),\cN(\bv))
=H^0\!\left(\mathcal X,
\mathcal H om_{\mathcal X}(\cF_i(M),\cN(\bv))\right),
\]
where the internal Hom is the equivariant sheafification of
\({}_S\Hom_R(F_i(M),N)(\bv)\), in (\ref{eq:exseq}).
This sequence together with (\ref{eq:Hom-seq})
and (\ref{eq:Hom-seq2}) form the following commutative diagram such
that all the horizontal and vertical sequences are exact.
\[
\begin{tikzcd}[row sep=1pc, column sep=1pc, cells={font=\small}]
 &  & H^{0}_{B}(_{S}\Hom_{R}(F_{0}(M),N))_{\bv}\ar[r]\ar[d] & H^{0}_{B}(_{S}\Hom_{R}(F_{1}(M),N))_{\bv}\ar[d]\\
0\ar[r]\ar[d] & \Hom_{R}(M,N)_{\bv}\ar[r]\ar[d,"\psi"] & \Hom_{R}(F_{0}(M),N)_{\bv}\ar[r]\ar[d,"\psi_{0}"] & \Hom_{R}(F_{1}(M),N)_{\bv}\ar[d,"\psi_{1}"]\\
0\ar[r] & \Hom_{\mathcal X}(\cM,\cN(\bv))\ar[r] & \Hom_{\mathcal X}(\cF_{0}(M),\cN(\bv))\ar[r]\ar[d] & \Hom_{\mathcal X}(\cF_{1}(M),\cN(\bv))\ar[d]\\
 &  & H^{1}_{B}(_{S}\Hom_{R}(F_{0}(M),N))_{\bv}\ar[r] & H^{1}_{B}(_{S}\Hom_{R}(F_{1}(M),N))_{\bv}
\end{tikzcd}
\]
To see that \(\psi\) is an isomorphism as desired, it is enough
to show that \(\psi_{0}\) and \(\psi_{1}\) are isomorphisms, thanks to
the 5-lemma.
In turn, to show the last condition, it is enough
to show
\begin{equation}\label{eq:loc-coh1}
H^{p}_{B}(_{S}\Hom_{R}(F_{i}(M),N))_{\bv}=0\quad(p,i\in\{0,1\},\,\bv\ge\be_{0}).
\end{equation}
Since
\[
H^{p}_{B}(\Hom_{R}(F_{i}(M),N))=\bigoplus^{b_{i}(M)}_{j=1}H^{p}_{B}(N)(\ba_{i,j}(M)),
\]
it is reduced to show
\begin{equation}\label{eq:loc-coh1.5}
H^{p}_{B}(N)_{\bv}=0\quad
\left(p\in\{0,1\},\ \bv\ge\be_{0}+\underline{\ba}_{i}(M),\
i\in\{0,1\}\text{ with }F_i(M)\ne0\right).
\end{equation}
In the rest of the proof, conditions indexed by \(i\in\{0,1\}\) are
likewise imposed only when \(F_i(M)\ne0\).
An obstacle to continuing to mimic Smith's argument \cite{smith2000computing}
is that \(B\) is not {*}local in the sense of \cite[14.1.1]{brodmann2012localcohomology}
so that the local duality \cite[14.4.1]{brodmann2012localcohomology}
does not apply to the local cohomology groups \(H^{p}_{B}(N)\).

\textbf{Step 2:} We reduce (\ref{eq:loc-coh1.5}) to similar vanishing
results for local cohomology groups relative to different ideals so
that local duality can be applied.
For \(s\in\{1,2\}\), we define the
ideal
\[
B_{s}=(x_{s,0},\dots,x_{s,d_{s}})\subset S.
\]
Then \(B=B_{1}B_{2}\). We have the (graded) Mayer--Vietoris sequence
\cite[Theorem 3.2.3 and Exercise 14.1.5]{brodmann2012localcohomology}:
\begin{align*}
0 & \to H^{0}_{B_{1}+B_{2}}(N)\to H^{0}_{B_{1}}(N)\oplus H^{0}_{B_{2}}(N)\to H^{0}_{B_{1}B_{2}}(N)\to\\
 & \to H^{1}_{B_{1}+B_{2}}(N)\to H^{1}_{B_{1}}(N)\oplus H^{1}_{B_{2}}(N)\to H^{1}_{B_{1}B_{2}}(N)\to\\
 & \to H^{2}_{B_{1}+B_{2}}(N)\to\cdots
\end{align*}
Thus, to show (\ref{eq:loc-coh1.5}), it is enough to show
\begin{gather}\label{eq:loc-coh2}
H^{p}_{B_{s}}(N)_{\bv}=0\quad(p\in\{0,1\},\,s\in\{1,2\},\,\bv\ge\be_{0}+\underline{\ba}_{i}(M)\,(i\in\{0,1\}))
\end{gather}
and
\begin{gather}\label{eq:loc-coh3}
H^{p}_{B_{1}+B_{2}}(N)_{\bv}=0\quad(p\in\{1,2\},\,\bv\ge\be_{0}+\underline{\ba}_{i}(M)\,(i\in\{0,1\})).
\end{gather}

\textbf{Step 3:} We show (\ref{eq:loc-coh2}) by using local duality
and reducing the problem to a vanishing result for some Ext groups.
First consider the case \(s=1\). Let
\[
\overline{B}_{1}:=(x_{1,0},\dots,x_{1,d_{1}})\subset S_{1}:=k[x_{1,0},\dots,x_{1,d_{1}}].
\]
We consider the decomposition \(N=\bigoplus_{w\in\ZZ}N_{\bullet,w}\)
of \(N\), where \(N_{\bullet,w}:=\bigoplus_{v\in\ZZ}N_{v,w}\). From
\cite[Example 3.6.10]{bruns1998cohenmacaulay} (see also \cite[Example 4.3(b)]{barile2025textbackslashmathbbztextasciicircumrgraded}),
the {*}canonical module of \(S_{1}\) is \(S_{1}(-c_{1})\). From the
local duality \cite[14.4.1]{brodmann2012localcohomology}, we have
graded isomorphisms
\begin{equation}\label{eq:loc-coh4}
\begin{aligned}H^{p}_{B_{1}}(N) & \cong H^{p}_{\overline{B}_{1}}(_{S_{1}}N) & (\text{op.\,cit., 4.2.1})\\
 & \cong H^{p}_{\overline{B}_{1}}\left(\bigoplus_{w\in\ZZ}N_{\bullet,w}\right)\\
 & \cong\bigoplus_{w\in\ZZ}H^{p}_{\overline{B}_{1}}\left(N_{\bullet,w}\right) & (\text{op.\,cit., 3.4.10})\\
 & \cong\bigoplus_{w\in\ZZ}\Hom_{k}(\Ext^{d_{1}+1-p}_{S_{1}}(N_{\bullet,w},S_{1})(-c_{1}),k). & (\text{op.\,cit., 14.4.1})
\end{aligned}
\end{equation}
In particular, for each \((v,w)\in\ZZ^{2}\), we have the following
implication:
\begin{equation}\label{eq:impli}
\Ext^{d_{1}+1-p}_{S_{1}}(N_{\bullet,w},S_{1})_{-v-c_{1}}=0\Rightarrow H^{p}_{B_{1}}(N)_{v,w}=0.
\end{equation}
To estimate nonvanishing degrees of \(\Ext^{d_{1}+1-p}_{S_{1}}(N_{\bullet,w},S_{1})\),
we get a free resolution of \(N_{\bullet,w}\) by restricting \(F_{\bullet}(_{S}N)\)
to degrees in \(\ZZ\times\{w\}\):
\begin{multline*}
F_{\bullet}(_{S}N)_{\bullet,w}\colon\cdots\to\bigoplus^{b_{1}(_{S}N)}_{j=1}\bigoplus_{\bx\in\fM(w-a_{2,1,j}(_{S}N))}S_{1}(-a_{1,1,j}(_{S}N))\bx\to\\
\to\bigoplus^{b_{0}(_{S}N)}_{j=1}\bigoplus_{\bx\in\fM(w-a_{2,0,j}(_{S}N))}S_{1}(-a_{1,0,j}(_{S}N))\bx\to0.
\end{multline*}
Here \(\fM(c)\) denotes the finite set of monomials in \(x_{2,0},\dots,x_{2,d_{2}}\)
of degree \(c\). We have
\begin{align*}
\Ext^{q}_{S_{1}}(N_{\bullet,w},S_{1}) & \cong H^{q}(\Hom_{S_{1}}(F_{\bullet}(_{S}N)_{\bullet,w},S_{1}))
\end{align*}
and
\[
\Hom_{S_{1}}(F_{q}(_{S}N)_{\bullet,w},S_{1})\cong\bigoplus^{b_{q}(_{S}N)}_{j=1}\bigoplus_{\fM(w-a_{2,q,j}(_{S}N))}S_{1}(a_{1,q,j}(_{S}N)).
\]
For the last isomorphism, we used the finiteness of the direct sum
(in fact, this is the reason why we restrict ourselves to degrees
in \(\ZZ\times\{w\}\) so that the direct sums involved become finite).
From these isomorphisms, we get
\[
\Ext^{q}_{S_{1}}(N_{\bullet,w},S_{1})_{-v-c_{1}}=0\quad(v+c_{1}\ge\overline{a}_{1,q}(_{S}N)+1).
\]
From (\ref{eq:impli}), we have
\begin{equation}\label{eq:van1}
H^{p}_{B_{1}}(N)_{\bv}=0
\end{equation}
if \(\bv=(v_{1},v_{2})\) satisfies \(v_{1}\ge\overline{a}_{1,d_{1}+1-p}(_{S}N)-c_{1}+1\),
in particular, if \(\bv\ge\overline{\ba}_{\bd+\bone-(p,p)}(_{S}N)-\bc+\bone\),
where \(\bc=(c_{1},c_{2})\) and \(c_{2}:=\sum_{t}c_{2,t}\). Similarly,
\begin{equation}\label{eq:van2}
H^{p}_{B_{2}}(N)_{\bv}=0\quad(\bv\ge\overline{\ba}_{\bd+\bone-(p,p)}(_{S}N)-\bc+\bone).
\end{equation}
For each \(p\in\{0,1\}\), we have
\begin{align*}
 & \be_{0}+\underline{\ba}_{0}(M)\\
 & =\bmax\{\overline{\ba}_{\bd+\bone}(_{S}N),\overline{\ba}_{\bd}(_{S}N),\overline{\ba}_{|\bd|+1}(_{S}N),\overline{\ba}_{|\bd|}(_{S}N)\}-\bc+\bone\\
 & \ge\overline{\ba}_{\bd+\bone-(p,p)}(_{S}N)-\bc+\bone.
\end{align*}
If \(F_1(M)=0\), there is no condition indexed by \(i=1\) to check.
Otherwise, minimality of the resolution gives
\(\underline{\ba}_{1}(M)\ge\underline{\ba}_{0}(M)\).  Thus
(\ref{eq:loc-coh2}) follows from (\ref{eq:van1}) and (\ref{eq:van2})
for \(p=0,1\).

\textbf{Step 4:} We show (\ref{eq:loc-coh3}) by an argument similar
to the last step. From \cite[Example 4.3(b)]{barile2025textbackslashmathbbztextasciicircumrgraded},
the {*}canonical module of \(S\) is \(S(-\bc)\) with \(\bc=(c_{1},c_{2})\)
as above. From the local duality \cite[14.4.1]{brodmann2012localcohomology},
we have
\begin{equation}\label{eq:loc-coh5}
H^{p}_{B_{1}+B_{2}}(N)=\Hom_{k}(\Ext^{|\bd|+2-p}_{S}(N,S)(-\bc),k)
\end{equation}
with \(|\bd|=d_{1}+d_{2}\). In particular, for each \(\bv\in\ZZ^{2}\),
we have the following implication:
\begin{equation}\label{eq:impli-2}
\Ext^{|\bd|+2-p}_{S}(N,S)_{-\bv-\bc}=0\Rightarrow H^{p}_{B_{1}+B_{2}}(N)_{\bv}=0.
\end{equation}
Since
\begin{align*}
\Ext^{q}_{S}(N,S) & =H^{q}(\Hom_{S}(F_{\bullet}(_{S}N),S))
\end{align*}
and
\[
\Hom_{S}(F_{q}(_{S}N),S)\cong\bigoplus^{b_{q}(_{S}N)}_{j=1}S(\ba_{q,j}(_{S}N)),
\]
we have
\[
\Ext^{q}_{S}(N,S)_{-\bv-\bc}=0\quad(\bv+\bc\ge\overline{\ba}_{q}(_{S}N)+\bone).
\]
From (\ref{eq:impli-2}), we have
\[
H^{p}_{B_{1}+B_{2}}(N)_{\bv}=0\quad(\bv\ge\overline{\ba}_{|\bd|+2-p}(_{S}N)-\bc+\bone).
\]
Since, for \(p\in\{1,2\}\), we have
\begin{align*}
 & \be_{0}+\underline{\ba}_{0}(M)\ge\overline{\ba}_{|\bd|+2-p}(_{S}N)-\bc+\bone,
\end{align*}
we obtain (\ref{eq:loc-coh3}), as desired. We have completed the
proof of the proposition.
\end{proof}

\begin{cor}
\label{cor:Hom2}
Let \(\br\in\ZZ^{2}\) be such that
\[
\br\ge\bmax\{\overline{\ba}_{\bd+\bone}(_{S}N),\overline{\ba}_{\bd}(_{S}N),\overline{\ba}_{|\bd|+1}(_{S}N),\overline{\ba}_{|\bd|}(_{S}N)\}-\bc+\bone.
\]
Then there is a natural isomorphism of bigraded \(R\)-modules
\[
\bigoplus_{\bv\in\NN^{2}}\Hom_{\mathcal X}(\cM,\cN(\bv))\cong\Hom_{R}(M_{\ge\br},N)_{\ge\bzero}.
\]
\end{cor}

\begin{proof}
The inclusion \(M_{\ge\br}\hookrightarrow M\) becomes an isomorphism
after localization at every \(x_{1,t_1}x_{2,t_2}\), so the two modules
determine the same coherent sheaf \(\cM\) on \(\mathcal X\).  If
\(M_{\ge\br}=0\), both sides of the asserted isomorphism are therefore
zero.  We may hence assume that \(M_{\ge\br}\ne0\).  We then have
\[
\underline{\ba}_{0}(M_{\ge\br})\ge\br\ge
\bmax\{\overline{\ba}_{\bd+\bone}(_{S}N),\overline{\ba}_{\bd}(_{S}N),
\overline{\ba}_{|\bd|+1}(_{S}N),\overline{\ba}_{|\bd|}(_{S}N)\}
-\bc+\bone.
\]
Consequently, the bound \(\be_0\) of Proposition \ref{prop:Hom}, with
\(M\) replaced by \(M_{\ge\br}\), satisfies
\[
\be_{0}\le\br-\underline{\ba}_{0}(M_{\ge\br})\le\bzero.
\]
The proposition is therefore applicable with \(\be=\bzero\), proving
the corollary.
\end{proof}

\section{Stein factorizations and contraction morphisms\label{sec:Stein-factorization}}

In this section, we give an algorithm to compute the Stein factorization
of a morphism of projective varieties and apply it to find contraction
morphisms.

\subsection{Computing a homogeneous coordinate ring of \(Z\)\label{subsec:Computing-homog-coord}}

Let \(f\colon Y\to X\) be a graph morphism of monograded varieties.
Replacing \(X\) with the image of \(f\) if necessary, we assume without
loss of generality that \(f\) is surjective. Suppose that \(Y\) and
\(X\) are written as
\begin{gather*}
Y=\Proj B,\quad B=k[y_{0},\dots,y_{n}]/\fq,\\
X=\Proj A,\quad A=k[x_{0},\dots,x_{m}]/\fp.
\end{gather*}
The graph \(\Gamma_{f}\) of \(f\) is written as
\[
\Gamma_{f}=\biProj R,\quad R=(B\otimes A)/\ft.
\]
As in Section \ref{sec:Bigraded-global-Hom}, we consider the bigraded
polynomial ring
\begin{align}\label{eq:bigr-poly}
S & =k[y_{0},\dots,y_{n},x_{0},\dots,x_{m}].
\end{align}
We can rewrite \(R\) as \(S/\widetilde{\ft}\), where \(\widetilde{\ft}\)
is the preimage of \(\ft\) in \(S\).

The Stein factorization
\[
Y\xrightarrow{h}Z\xrightarrow{g}X
\]
of \(f\) is given by setting \(Z=\cSpec_{X}\,f_{*}\cO_{Y}\) (see the
proof of \cite[Chapter III, Corollary 11.5]{hartshorne1977algebraic}).
Put
\[
\mathcal X:=
\left[(\Spec A\setminus V(A_{\dagger}))/\GG_m\right]
\]
and
\[
\mathcal G_f:=
\left[(\Spec R\setminus V(R_{\dagger}))/\GG_m^2\right].
\]
Their coarse moduli spaces are \(X\) and \(\Gamma_f\), respectively.
The second projection restricts to a morphism
\[
\pr_2\colon\Gamma_f\to X,
\]
which is identified with \(f\colon Y\to X\) under the isomorphism
\(\Gamma_f\cong Y\).  The inclusion \(A\to R\) and the second projection
\(\GG_m^2\to\GG_m\) induce its lift to a proper morphism
\[
\widetilde{\pr}_2
\colon\mathcal G_f\to\mathcal X
\]
of quotient stacks.
Let
\[
\mathcal F:=(\widetilde{\pr}_2)_*
\cO_{\mathcal G_f}
\]
and define the graded ring
\begin{align*}
C
&:=\bigoplus_{v\in\NN}
H^0\bigl(\mathcal X,\mathcal F\otimes\cO_{\mathcal X}(v)\bigr)\\
&=\bigoplus_{v\in\NN}
H^0\bigl(\mathcal G_f,\cO_{\mathcal G_f}(0,v)\bigr).
\end{align*}
The pullback along the second projection
\(\GG_m^2\to\GG_m\) induces the inclusion of character
lattices \(\ZZ\to\ZZ^2\), \(v\mapsto(0,v)\).  Hence
\[
\widetilde{\pr}_2^*\cO_{\mathcal X}(v)
\cong\cO_{\mathcal G_f}(0,v).
\]
The projection formula gives
\[
\mathcal F\otimes\cO_{\mathcal X}(v)
\cong
(\widetilde{\pr}_2)_*
\cO_{\mathcal G_f}(0,v),
\]
which proves the second equality above after taking global sections.
The multiplication on \(\mathcal F\) makes \(C\) a graded \(A\)-algebra,
with its natural homogeneous map \(\phi\colon A\to C\).
\begin{lem}
\label{lem:mor-Proj}
The natural map \(\phi\colon A\to C\) of graded
rings induces the morphism \(g\colon Z\to X\) appearing in the Stein
factorization of \(f\).
\end{lem}

\begin{proof}
Let \(\phi\colon A\to C\) be the natural homogeneous map defined above,
and let \(l:=\lcm(\deg x_0,\dots,\deg x_m)\).  The nonempty opens
\[
D_+(x_i^{l/\deg x_i})\subset X
\]
form an affine cover; more generally, write \(U_\alpha:=D_+(\alpha)\)
for one of these opens, with \(\alpha\in A_l\), and let
\(\mathcal U_\alpha\subset\mathcal X\) be the corresponding open substack.  The line bundle \(\cO_{\mathcal X}(l)\)
descends to the ample line bundle \(\cO_X(l)\).  Let
\(p\colon\mathcal X\to X\) be the coarse moduli map and put
\(\mathcal E:=p_*\mathcal F\).  Since
\(\cO_{\mathcal X}(nl)\cong p^*\cO_X(nl)\), the projection formula gives
\[
C_{nl}=H^0\bigl(X,\mathcal E\otimes\cO_X(nl)\bigr).
\]
We apply the usual Serre localization argument not on the stack but
on the scheme \(X\), to the sheaf \(\mathcal E\) and the ample line bundle
\(\cO_X(l)\).  The standard proof of
\cite[Chapter II, Proposition 5.11]{hartshorne1977algebraic}, applied
to the \(l\)-th Veronese of \(C\), gives
\[
C_{(\phi(\alpha))}=H^0(U_\alpha,\mathcal E).
\]
Since \(\mathcal U_\alpha=p^{-1}(U_\alpha)\), the definition of direct
image gives
\[
H^0(U_\alpha,\mathcal E)
=H^0(\mathcal U_\alpha,\mathcal F).
\]
Finally, the definition
\(\mathcal F=(\widetilde{\pr}_2)_*
\cO_{\mathcal G_f}\) gives
\begin{equation}\label{eq:stack-Stein-localization}
C_{(\phi(\alpha))}
=H^0(\mathcal U_\alpha,\mathcal F)
=H^0(\widetilde{\pr}_2^{-1}(\mathcal U_\alpha),
\cO_{\mathcal G_f}).
\end{equation}

The stacks \(\mathcal X\) and \(\mathcal G_f\) are tame in characteristic
zero.  For the coarse moduli map
\(\pi\colon\mathcal G_f\to\Gamma_f\) we therefore have
\(\pi_*\cO_{\mathcal G_f}=\cO_{\Gamma_f}\)
\cite[Theorem 3.2]{abramovich2008tame}.  The coarse space of
\(\widetilde{\pr}_2^{-1}(\mathcal U_\alpha)\) is
\(f^{-1}(U_\alpha)\), as is also seen directly from the invariant rings
on these quotient charts.  Hence \eqref{eq:stack-Stein-localization}
continues as
\[
C_{(\phi(\alpha))}
=H^0(f^{-1}(U_\alpha),\cO_Y)
=H^0(U_\alpha,f_*\cO_Y).
\]
These isomorphisms commute with further localization, so they identify
the sheaf associated with the graded \(A\)-module \(C\) with
\(f_*\cO_Y\).  Under this identification the localized map
\[
A_{(\alpha)}\to C_{(\phi(\alpha))}
\]
is the natural map
\(\cO_X(U_\alpha)\to(f_*\cO_Y)(U_\alpha)\).  Thus
\(\Proj C\to\Proj A\) and the morphism
\[
g\colon\cSpec_X f_*\cO_Y\to X
\]
are obtained by gluing the same
morphisms of affine schemes, proving the lemma.
\end{proof}

\begin{rem}
The quotient stacks are used here to retain all grading indices; they
are not intrinsic to the Stein factorization. With \(l\) as in the proof,
one could instead work on the coarse spaces and use only the degrees
divisible by \(l\). Indeed, \(\cO_{\mathcal X}(nl)\) descends to the line
bundle \(\cO_X(nl)\), and the resulting construction uses the Veronese
subrings
\[
A^{(l)}:=\bigoplus_{n\geq0}A_{nl}
\quad\text{and}\quad
C^{(l)}:=\bigoplus_{n\geq0}C_{nl}.
\]
Since taking these Veronese subrings does not change the \(\Proj\) of either ring,
this gives the same morphism \(Z\to X\) without using quotient stacks.
We use the stacky formulation
because \(\cO_{\mathcal X}(v)\) is a line bundle for every \(v\), so that
we may retain every graded piece \(C_v\). Passing to multiples of \(l\) increases the degrees occurring in the computation;
retaining the smallest possible indices is generally more efficient.
\end{rem}

We now apply Corollary \ref{cor:Hom2}, in its quotient-stack form, to
the bigraded rings \(S\) and \(R\) defined in this section, with \(M=N=R\).
Let \(\br\in\ZZ^2\) be as in that corollary.  Then
\begin{align*}
C
&=\bigoplus_{v\ge0}
H^0(\mathcal G_f,\cO_{\mathcal G_f}(0,v))\\
&=\bigoplus_{v\ge0}
\Hom_{\mathcal G_f}
(\cO_{\mathcal G_f},\cO_{\mathcal G_f}(0,v))\\
&=\bigoplus_{v\ge0}\Hom_R(R_{\ge\br},R)_{0,v}\\
&=:\Hom_R(R_{\ge\br},R)_{0,\ge0}.
\end{align*}
The isomorphism is induced by equivariant sheafification and hence is
compatible with multiplication.  Thus the last module carries the
graded \(A\)-algebra structure on \(C\), not merely its \(A\)-module
structure.
To compute this graded ring, we first compute \(\Hom_{R}(R_{\ge\br},R)_{\ge\bzero}\)
as an \(R\)-module and get a free presentation
\[
G\to F\to\Hom_{R}(R_{\ge\br},R)_{\ge\bzero}\to0,
\]
where \(G\) and \(F\) are free \(R\)-modules and maps are degree-preserving.
Restricting this to the part of degrees in \(\{0\}\times\NN\), we get
an exact sequence
\[
G_{0,\ge0}\to F_{0,\ge0}\to\Hom_{R}(R_{\ge\br},R)_{0,\ge0}\to0
\]
of graded modules over \(R_{0,\ge0}=A\). Next, we compute a surjection
\(G'\to G_{0,\ge0}\) from a free \(A\)-module \(G'\) and a free presentation
of the \(A\)-module \(F_{0,\ge0}\),
\begin{gather*}
F''\to F'\to F{}_{0,\ge0}\to0.
\end{gather*}
We then lift the map \(G_{0,\ge0}\to F_{0,\ge0}\) to a map \(G'\to F'\).
The induced sequence
\[
G'\oplus F''\to F'\to\Hom_{R}(R_{\ge\br},R)_{0,\ge0}\to0
\]
is a free presentation of \(C=\Hom_{R}(R_{\ge\br},R)_{0,\ge0}\) as
a graded \(A\)-module.

Next we describe \(C\) as an \(A\)-algebra. We fix a bihomogeneous nonzero
element \(\gamma\in R_{\ge\br}\), say the image of some monomial in
\(S\). The localization \(R_{\gamma}\) of the ring \(R\) by \(\gamma\),
explicitly presented as \(R[u]/\langle u\gamma-1\rangle\), has a natural
\(\ZZ^{2}\)-grading with \(\deg(\gamma^{-1})=-\deg\gamma\).
\begin{lem}
\label{lem:iota}
The map
\[
\iota\colon\Hom_{R}(R_{\ge\br},R)_{0,\ge0}\to R_{\gamma},\quad\psi\mapsto\psi(\gamma)/\gamma
\]
is a degree-preserving injective ring homomorphism.
\end{lem}

\begin{proof}
Under the isomorphism established above, a homogeneous element
\(\psi\in\Hom_R(R_{\ge\br},R)_{0,v}\) corresponds to a section
\[
s_\psi\in H^0(\mathcal G_f,\cO_{\mathcal G_f}(0,v)).
\]
Let \(\mathcal D(\gamma)\subset\mathcal G_f\) be the open substack on
which \(\gamma\) is nonzero.  Equivariant sheafification identifies
the restriction of \(s_\psi\) to this open with the homogeneous
fraction
\[
\frac{\psi(\gamma)}{\gamma}\in R_\gamma.
\]
Thus \(\iota\) is restriction of sections, followed by the natural
embedding of the corresponding homogeneous part into \(R_\gamma\).
Since \(R\) is a domain and \(\gamma\ne0\), the open
\(\mathcal D(\gamma)\) is nonempty and dense in the integral stack
\(\mathcal G_f\).  A section of a line bundle that vanishes there
vanishes everywhere, so \(\iota\) is injective.  Restriction respects
products, and hence \(\iota\) is a ring homomorphism.  Finally, if
\(\deg\gamma=(a,b)\) and \(\deg\psi=(0,v)\), then
\[
\deg\frac{\psi(\gamma)}{\gamma}
=((a,b)+(0,v))-(a,b)=(0,v),
\]
so the map preserves degrees.
\end{proof}

If \(\psi_{1},\dots,\psi_{l}\) are homogeneous generators of \(C\) as
an \(A\)-module, then we have
\[
C=A[\psi_{1}(\gamma)/\gamma,\dots,\psi_{l}(\gamma)/\gamma]\subset R_{\gamma}.
\]
Thus, we get an explicit presentation of \(C\) by computing the kernel
of the \(A\)-algebra homomorphism
\[
A[z_{1},\dots,z_{l}]\to R_{\gamma},\quad z_{i}\mapsto\psi_{i}(\gamma)/\gamma.
\]
In summary, we can get an explicit presentation of the graded ring
\(C\) as an \(A\)-algebra.

The degree-zero part of \(C\) is \(H^{0}(X,f_{*}\cO_{Y})=H^{0}(Y,\cO_{Y})\),
a finite extension field of \(k\), and it need not be \(k\) itself once \(k\)
is not algebraically closed. The homogeneous coordinate ring of a
monograded variety, on the other hand, has degree-zero part \(k\). We
therefore pass to the following subring, which changes neither
\(\Proj C\) nor the morphism to \(X\).
\begin{lem}
\label{lem:section-ring-over-k}
Put \(L:=C_{0}\) and
\[
C^{[k]}:=k\oplus\bigoplus_{v>0}C_{v}\subset C.
\]
Then \(C^{[k]}\) is a graded subring of \(C\) with \(C^{[k]}_{0}=k\), finite
as an \(A\)-module, the homogeneous map \(A\to C\) factors through it, and
the inclusion \(C^{[k]}\subset C\) induces an isomorphism
\(\Proj C^{[k]}\cong\Proj C\) over \(X\). Moreover we can compute an
explicit presentation of \(C^{[k]}\).
\end{lem}

\begin{proof}
By construction, \(C^{[k]}\) is a graded subring of \(C\). Since
\(A_{0}=k\), it is also a graded \(A\)-submodule of \(C\). The ring \(A\) is
Noetherian and
\(C\) is a finite graded \(A\)-module by the free presentation computed
above, so \(C^{[k]}\) is a finite \(A\)-module. The map \(A\to C\) carries
\(A_{0}=k\) into \(C^{[k]}_{0}\) and \(A_{d}\) into \(C_{d}=C^{[k]}_{d}\) for
\(d>0\), and hence factors through \(C^{[k]}\).
The rings \(C^{[k]}\) and \(C\) have the same part of positive degree, so
the two graded rings have the same \(\Proj\), and the factorization just
found shows this isomorphism is one over \(X\).

For the last assertion, recall that \(C=A[\psi_{1}(\gamma)/\gamma,\dots,\psi_{l}(\gamma)/\gamma]\).
The ring \(A\) is generated over \(A_{0}=k\) by homogeneous elements of
positive degree, so as a \(k\)-algebra \(C\) is generated by the images of
those together with the elements \(\psi_{i}(\gamma)/\gamma\). We sort this
list of generators by degree.

Those \(\psi_{i}(\gamma)/\gamma\) of degree zero lie in \(C_{0}=L\), and
they span \(L\) over \(k\): the \(\psi_{i}\) generate \(C\) as an \(A\)-module and
\(A\) has no part of negative degree, so
\(C_{0}=\sum_{\deg\psi_{i}=0}A_{0}\cdot\psi_{i}(\gamma)/\gamma\). Let
\(c_{1},\dots,c_{r}\) be the generators of positive degree, that is, the
images of the generators of \(A\) together with those
\(\psi_{i}(\gamma)/\gamma\) of positive degree. The generators of degree
zero having been absorbed into \(L\), we obtain
\[
C=L[c_{1},\dots,c_{r}].
\]

Since the \(\psi_{i}(\gamma)/\gamma\) of degree zero span \(L\) over
\(k\), using linear algebra over \(k\) we compute, from the
presentation of \(C\) obtained above, a basis
\(e_{1}=1,e_{2},\dots,e_{s}\) of \(L\) with \(e_{2},\dots,e_{s}\) among
them. Then
\[
C^{[k]}=k[e_{i}c_{j}\mid1\le i\le s,\,1\le j\le r].
\]
Indeed, every \(c_{j}\) has positive degree, so \(C^{[k]}\) is spanned over
\(k\) by \(1\) together with the elements \(e_{i}c_{j_{1}}\cdots c_{j_{t}}\)
with \(t\ge1\), and such an element is the product
\((e_{i}c_{j_{1}})(e_{1}c_{j_{2}})\cdots(e_{1}c_{j_{t}})\) of generators.
An explicit presentation of \(C^{[k]}\) is obtained by computing
the kernel of the \(k\)-algebra homomorphism
\[
k[z_{ij}\mid1\le i\le s,\,1\le j\le r]\to R_{\gamma},\quad
z_{ij}\mapsto e_{i}c_{j},
\]
as above.
\end{proof}

\subsection{Computing the graph of \(h\colon Y\to Z\)\label{subsec:Computing-graph}}

We have the following diagram, where the solid arrows are natural ones and the dashed arrow is the one induced by them:
\[
\begin{tikzcd}[row sep=2.5em, column sep=2.5em]
 & \Gamma_{f}\ar[dl,"\sim",sloped]\ar[d,dashed]\ar[dr] & \\
Y & Y\times Z\ar[l]\ar[r] & Z
\end{tikzcd}
\]
The middle vertical arrow is a closed immersion and its
image is \(\Gamma_{h}\).
This diagram of schemes induces the following diagram of ring homomorphisms:
\[
\begin{tikzcd}[row sep=2.5em, column sep=2.5em]
 & R_{\gamma} & \\
B\ar[r]\ar[ur] & B\otimes C^{[k]}\ar[u,"\Psi"] & C^{[k]}\ar[l]\ar[ul,"\iota"']
\end{tikzcd}
\]
Here \(B\) carries bidegrees \((*,0)\) and \(C^{[k]}\) bidegrees \((0,*)\), and
both slanted arrows preserve bidegrees, so \(\Psi\) does as well and
\(\ker\Psi\) is bihomogeneous. It is also prime, because \(R_{\gamma}\) is
a domain. Moreover, it does not contain the irrelevant ideal: that ideal
is generated by \(y\otimes c\) with \(y\in B\) and \(c\in C^{[k]}\)
homogeneous of positive degrees, whereas, for nonzero such \(y\) and \(c\),
the injectivity of the slanted arrows and the domain property give
\(\Psi(y\otimes c)=\bar y\,\iota(c)\ne0\).

The target \(R_{\gamma}\) is a ring of sections
over \(D_{+}(\gamma)\), so a bihomogeneous element
\(F\in B\otimes C^{[k]}\) lies in \(\ker\Psi\) exactly when the corresponding
section vanishes on the image of that open subset; as \(D_{+}(\gamma)\) is dense
in the integral variety \(\Gamma_{f}\), its image is dense in
\(\Gamma_{h}\) and that is the same as vanishing on all of \(\Gamma_{h}\).
Therefore, \(\ker\Psi\) is the bihomogeneous prime ideal defining
\(\Gamma_{h}\).

We thus obtain the following proposition:

\begin{prop}
\label{prop:Stein}
There is an algorithm, Algorithm \ref{algo:Stein},
that computes the Stein factorization of a graph morphism between
monograded varieties.
\end{prop}

\begin{algorithm}[ht]
\raggedright
\caption{Computing the Stein factorization}
\label{algo:Stein}

\textbf{Input:} Monograded varieties \(Y=\Proj B,X=\Proj A\) and a
surjective graph morphism \(f\colon Y\to X\).

\textbf{Output:} A monograded variety \(Z\), a homogeneous morphism
\(g\colon Z\to X\) and a graph morphism \(h\colon Y\to Z\) such that
\(Y\xrightarrow{h}Z\xrightarrow{g}X\) is the Stein factorization of
\(f\).

\textbf{Procedure:}
\begin{enumerate}
\item Writing the graph \(\Gamma_{f}\) as \(\Gamma_{f}=\biProj R\), with \(R\)
a quotient ring of \(B\otimes A\), compute the graded ring
\(C=\Hom_{R}(R_{\ge\br},R)_{0,\ge0}\) together
with a homogeneous map \(A\to C\), and pass to the homogeneous coordinate
ring \(C^{[k]}=k\oplus C_{>0}\) of \(Z\) (Lemma
\ref{lem:section-ring-over-k}). Here \(\br\) is the pair of integers
in Corollary \ref{cor:Hom2} applied to the situation of Section \ref{subsec:Computing-homog-coord}.
\item Compute the kernel of
\[
\Psi\colon B\otimes C^{[k]}\to R_{\gamma},
\]
the map induced by \(B\to R\to R_{\gamma}\) and by \(\iota\) on the
presentation of \(C^{[k]}\) obtained in Step 1, and return
\(Z=\Proj C^{[k]}\), the homogeneous morphism \(g\colon Z\to X\) of Step 1
and the graph morphism \(h\colon Y\to Z\) given by
\(\Gamma_{h}=\biProj\bigl((B\otimes C^{[k]})/\ker\Psi\bigr)\).
\end{enumerate}
\end{algorithm}

\subsection{Finding contraction morphisms\label{subsec:finding-contractions}}

We now apply the Stein factorization algorithm to the morphism defined by
the nef threshold computed in Section \ref{sec:nefness-canonical}.

\begin{prop}
\label{prop:contraction-from-lambda}
Suppose that we are given a normal \(\QQ\)-Gorenstein monograded variety
\(X\) of dimension \(d\) with only log terminal singularities such that
\(K_X\) is not nef, an ample Cartier divisor \(H\) on \(X\), and a positive
integer \(a\) such that \(aK_X\) is Cartier. Put
\[
\lambda:=\inf\{t\in\RR_{\ge0}\mid K_X+tH\text{ is nef}\}
\quad\text{and}\quad
D:=K_X+\lambda H.
\]
Then we can
compute a positive integer \(M\) such that \(MD\) is Cartier and
base-point-free, the associated morphism \(\Phi_{|MD|}\) and its Stein factorization
\(\phi\colon X\to Z\), all as graph morphisms of monograded varieties.
Moreover \(\phi\) contracts at least one curve, and every curve \(C\)
contracted by \(\phi\) satisfies \(D\cdot C=0\) and \(K_{X}\cdot C<0\).
\end{prop}

\begin{proof}
We compute \(\lambda\) by Proposition \ref{prop:lambda-computable} and
write \(\lambda=p/q\) with positive integers \(p\) and \(q\). Then \(N:=aq\)
satisfies that \(ND\) is Cartier, and \(M:=mN\) with \(m\) as in Proposition
\ref{prop:nef-scaled} is as desired: the divisor \(D\) is nef by Lemma
\ref{lem:lambda}(2), so \(|MD|\) is base-point-free by Proposition
\ref{prop:nef-scaled}. We can therefore compute \(\Phi_{|MD|}\colon X \to \PP^n\) by Lemma
\ref{lem:basept-free} and its Stein factorization \(\phi\colon X\to Z\) by
Proposition \ref{prop:Stein}.

Since \(Z\to\PP^{n}\) is finite, a curve of \(X\) is contracted by \(\phi\) if
and only if it is contracted by \(\Phi_{|MD|}\). As \(D\) is not ample by
Lemma \ref{lem:lambda}(2) but \(MD\) is base-point-free,
\(\Phi_{|MD|}\), and hence \(\phi\), contracts at least one curve; moreover,
a curve \(C\) is contracted if and only if \(D\cdot C=0\). For such a curve
we have
\(K_{X}\cdot C=-\lambda H\cdot C<0\), because \(\lambda>0\) and \(H\) is ample.
\end{proof}

\begin{rem}
Let \(X\), \(H\), \(\lambda\), \(D=K_{X}+\lambda H\) and \(\phi\colon X\to Z\) be
as in Proposition \ref{prop:contraction-from-lambda} and put
\[
F:=\{z\in\NE(X)\mid D\cdot z=0\},
\]
where \(\NE(X)\) is the closed cone of curves, that is, the Mori
cone. Then \(F\) is a nonzero \(K_X\)-negative extremal face, and
\(\phi\) contracts exactly the curves whose classes lie in \(F\).
Moreover, \(\phi_*\cO_X=\cO_Z\) and \(-K_X\) is \(\phi\)-ample, so
\(\phi\) is the contraction of \(F\) in the sense of
\cite{kollar2021relative}. Here no \(\QQ\)-factoriality is assumed, and
\(F\) need not be a ray. Consequently, the preceding constructions give
Algorithm \ref{algo:find-contraction} for finding such a contraction.
\end{rem}

\begin{algorithm}[ht]
\raggedright
\caption{Finding a contraction morphism}
\label{algo:find-contraction}

\textbf{Input:} A normal \(\QQ\)-Gorenstein log terminal monograded
variety \(X\) such that \(K_{X}\) is not nef, and a positive integer \(a\)
such that \(aK_{X}\) is Cartier.

\textbf{Output:} A graph morphism \(\phi\colon X\to Z\) of normal
monograded varieties that is the contraction of a \(K_{X}\)-negative
extremal face of the Mori cone \(\NE(X)\).

\textbf{Procedure:}
\begin{enumerate}
\item Compute an ample Cartier divisor \(H\) on \(X\) (Lemma
\ref{lem:ample-cartier}).
\item Compute \(\lambda=\inf\{t\in\RR_{\ge0}\mid K_{X}+tH\text{ is nef}\}\)
(Algorithm \ref{algo:lambda}) and put \(D:=K_{X}+\lambda H\).
\item Compute a positive integer \(M\) such that \(MD\) is Cartier and
base-point-free, and the morphism \(\Phi_{|MD|}\colon X\to\PP^{n}\)
(Proposition \ref{prop:contraction-from-lambda}).
\item Compute the Stein factorization \(X\xrightarrow{\phi}Z\to\PP^{n}\)
of \(\Phi_{|MD|}\) (Algorithm \ref{algo:Stein}) and return
\(\phi\colon X\to Z\).
\end{enumerate}
\end{algorithm}

\section{Relative canonical models and the minimal model program\label{sec:relative-canonical-models}}

In this section, we provide an algorithm for computing relative canonical
models. Combining it with algorithms obtained in earlier sections, we obtain
an algorithm for the minimal model program in dimension three.

\subsection{Relative canonical models via Rees-algebra candidates}
\label{subsec:canonical-model-candidates}

Let \(Y\) be a normal projective variety with only log terminal
singularities and let \(f\colon Y\to X\) be a projective birational
morphism with \(f_{*}\cO_{Y}=\cO_{X}\) such that \(-K_{Y}\) is \(f\)-ample.
Recall that a normal variety \(V\) is said to be \emph{of klt type} if there
is an effective \(\QQ\)-divisor \(\Delta\) such that \(K_V+\Delta\) is \(\QQ\)-Cartier and \((V,\Delta)\) is klt.

\begin{lem}
\label{lem:target-klt-type}
The variety \(X\) is of klt type.
\end{lem}

\begin{proof}
The pair \((Y,0)\) is klt and \(-K_{Y}\) is \(f\)-ample, so \(Y\) is of Fano
type over \(X\) in the sense of
\cite[Lemma--Definition~2.6(i)]{prokhorov2009towards}, the relative notion
being the one introduced in the sentence that follows that
lemma-definition. By \cite[Lemma~2.8(i)]{prokhorov2009towards}, applied
over the base \(X\) and in the birational case, \(X\) is then of Fano type
over itself. In particular, there is an effective \(\QQ\)-divisor
\(\Delta_{X}\) such that \((X,\Delta_{X})\) is klt, so \(X\) is of klt type.
\end{proof}

\begin{lem}
\label{lem:self-relative-canonical-model}
Let \(V\) be a normal projective variety of klt type. Then the canonical
algebra
\[
\bigoplus_{n\ge0}\cO_{V}(nK_{V})
\]
is a finitely generated \(\cO_{V}\)-algebra, and the variety
\[
V^{\mathrm{c}}:=\cProj_{V}\bigoplus_{n\ge0}\cO_{V}(nK_{V})
\]
is normal and \(\QQ\)-Gorenstein and has only log terminal
singularities. The natural morphism \(g\colon V^{\mathrm{c}}\to V\)
is small, and \(K_{V^{\mathrm{c}}}\) is \(g\)-ample. Moreover \(g\)
is the only normal small birational modification of \(V\) such that
the strict transform of \(K_{V}\) becomes \(\QQ\)-Cartier and
relatively ample.
\end{lem}
\begin{proof}
Apply \cite[Lemma~2.5]{zhuang2024direct}, proved in
\cite[Lemma~4.7]{zhuang2024boundedness}, to the Weil divisor \(K_{V}\).
It gives the finite generation of \(\bigoplus_{n\ge0}\cO_{V}(nK_{V})\).
The construction in the proof of the cited lemma realizes
\(V^{\mathrm{c}}\) as the log canonical model of a klt pair, so
\(V^{\mathrm{c}}\) is normal and has only log terminal
singularities.
In the proof of \cite[Lemma~4.7]{zhuang2024boundedness} it is shown that
for a sufficiently divisible \(r\) the morphism
\(\cProj_{V}\bigoplus_{n\ge0}\cO_{V}(nrK_{V})\to V\) is the unique normal small
birational modification on which the strict transform of \(K_{V}\) is
\(\QQ\)-Cartier and relatively ample. This morphism is canonically
isomorphic to \(g\).
\end{proof}

\begin{defn}
\label{defn:relative-canonical-model}
Let \(f\colon Y\to X\) be a projective morphism of normal varieties, and
write \(\cO_Y(nK_Y)\) for the reflexive divisorial sheaf associated with
\(nK_Y\). If the relative canonical algebra
\[
\bigoplus_{n\ge0}f_*\cO_Y(nK_Y)
\]
is finitely generated as an \(\cO_X\)-algebra, the structural morphism
\[
\cProj_X\bigoplus_{n\ge0}f_*\cO_Y(nK_Y)\to X
\]
is called the \emph{relative canonical model of \(Y\) over \(X\)}.
\end{defn}

In particular, Lemma \ref{lem:self-relative-canonical-model} says
that the relative canonical model of \(V\) over itself exists when
\(V\) is of klt type. It is the small morphism \(V^{\mathrm{c}}\to
V\) constructed there.

\begin{lem}
\label{lem:canonical-algebras-agree}
Let \(Y\) be a normal projective variety with only log terminal
singularities, and let \(f\colon Y\to X\) be a projective birational
morphism onto a normal projective variety such that
\(f_*\cO_Y=\cO_X\) and \(-K_Y\) is \(f\)-ample. Then the relative
canonical models of \(Y\) over \(X\) and of \(X\) over itself coincide.
More precisely, put
\[
\cA_X:=\bigoplus_{n\ge0}\cO_X(nK_X)
\quad\text{and}\quad
g\colon Z:=\cProj_X\cA_X\to X.
\]
Then for a sufficiently divisible positive integer \(r\) and every
\(n\ge0\),
\begin{equation}
\label{eq:source-target-canonical-algebras}
f_*\cO_Y(nrK_Y)=g_*\cO_Z(nrK_Z)=\cO_X(nrK_X).
\end{equation}
\end{lem}

\begin{proof}
The algebra \(\cA_X\) is finitely generated by Lemmas
\ref{lem:target-klt-type} and \ref{lem:self-relative-canonical-model}.
Let \(p\colon W\to Y\) and \(q\colon W\to Z\) be a common resolution,
choose the canonical divisors compatibly over \(X\), and set
\[
E:=p^*K_Y-q^*K_Z.
\]
The divisor \(E\) is \(q\)-exceptional. Moreover, \(-E\) is \(q\)-nef
because \(-K_Y\) is \(f\)-ample. The negativity lemma
\cite[Lemma~3.39]{kollar1998birational} therefore gives \(E\ge0\). Hence,
for a sufficiently divisible positive integer \(r\) and every \(n\ge0\),
the projection formula gives \eqref{eq:source-target-canonical-algebras};
here \(q_*\cO_W(nrE)=\cO_Z\) because \(nrE\) is effective and
\(q\)-exceptional. The identifications are compatible with multiplication,
so taking relative \(\Proj\) proves the assertion about the relative
canonical models.
\end{proof}

\begin{rem}
\label{rem:tests-miss-qfactoriality}
When \(f\) is small it is a flipping contraction in the sense of
\cite[Definition~4.8.2]{fujino2017foundations}, and \(g\) is its flip by
\cite[Lemma~4.8.3]{fujino2017foundations}. Neither of the two makes any
hypothesis on the relative Picard number.
\end{rem}

\begin{lem}
\label{lem:stack-Rees-candidate}
Let \(X=\Proj R\) be a normal monograded variety of klt type, and let
\(I\subset R\) be a homogeneous ideal with
\(I\cong\omega_R(-d)\) as graded \(R\)-modules, let \(m>0\), and let
\[
\pi^m\colon W^{(m)}:=\biProj R[I^{(m)}t]\to X
\]
be the bi-to-mono projection computed from the Rees algebra of the
\(m\)-th symbolic power \(I^{(m)}\) of \(I\). If \(\pi^m\) is small
and \(W^{(m)}\) satisfies \(S_2\), then \(\pi^m\) is the relative
canonical model of \(X\) over itself.
\end{lem}

\begin{proof}
Put
\[
\mathcal X=
[(\Spec R\setminus V(R_\dagger))/\GG_m]
\]
and let \(p\colon\mathcal X\to X\) be its coarse moduli map.  Denote by
\(\mathcal J_m\) the ideal on \(\mathcal X\) obtained by equivariant
sheafification of \(I^{(m)}\).  Equivariant sheafification commutes with
localization and with the double dual of a torsion-free module.  Hence
the graded isomorphism \(I\cong\omega_R(-d)\) gives
\begin{equation}\label{eq:stack-symbolic-canonical}
\mathcal J_m
\cong\mathcal K_m\otimes\cO_{\mathcal X}(-md),
\end{equation}
where \(\mathcal K_m\) denotes the reflexive sheaf on \(\mathcal X\)
corresponding to \(\cO_X(mK_X)\) in codimension one. In particular the
second factor in \eqref{eq:stack-symbolic-canonical} is a line bundle,
without any divisibility condition on \(m\).

Consider the stack blow-up
\[
q^m\colon\mathcal W^{(m)}
:=\cProj_{\mathcal X}\bigoplus_{n\ge0}\mathcal J_m^n
\to\mathcal X.
\]
Writing the construction on the quotient charts shows that its coarse
moduli space is \(W^{(m)}\): on each selected Rees chart the coarse
coordinate ring is the bidegree-zero part of the corresponding
localization, which is precisely the chart used in the definition of
\(\biProj R[I^{(m)}t]\).

To remove the grading shift in \(\mathcal J_m\), put
\[
\mathcal L_m:=\cO_{\mathcal W^{(m)}}(1)
\otimes(q^m)^*\cO_{\mathcal X}(md).
\]
This is \(q^m\)-ample.  Since the stabilizers are finite, some positive
power \(\mathcal L_m^{\otimes N}\) has trivial stabilizer action and
descends to a \(\pi^m\)-ample line bundle \(L_m\) on \(W^{(m)}\).

The hypotheses imply that \(W^{(m)}\) is normal: it is \(S_2\), and
it is \(R_1\) because \(X\) is normal and \(\pi^m\) is small.
Choose an open subset \(U\subset X\) whose complement has
codimension at least two and over which \(\pi^m\) is an isomorphism
and \(K_X\) is Cartier.  On \(U\), the line bundle \(\mathcal L_m\)
corresponds, by \eqref{eq:stack-symbolic-canonical}, to
\(\cO_U(mK_X)\).  After increasing \(N\) if necessary, descent
therefore gives
\[
L_m|_{(\pi^m)^{-1}(U)}
\cong
\cO_{W^{(m)}}(NmK_{W^{(m)}})|_{(\pi^m)^{-1}(U)}.
\]
Both sides extend as rank-one reflexive sheaves on the normal
variety \(W^{(m)}\), and the complement of \((\pi^m)^{-1}(U)\) has
codimension at least two.  Hence the isomorphism extends uniquely to
\begin{equation}\label{eq:descended-stack-Rees-line}
L_m\cong\cO_{W^{(m)}}(NmK_{W^{(m)}}).
\end{equation}
Thus \(K_{W^{(m)}}\) is \(\QQ\)-Cartier and \(\pi^m\)-ample.  The
uniqueness in Lemma \ref{lem:self-relative-canonical-model} now
shows that \(\pi^m\) is the relative canonical model of \(X\) over
itself.
\end{proof}

In the setting of this subsection, where \(f\colon Y\to X\) is the
projective birational morphism fixed above, Lemma
\ref{lem:canonical-algebras-agree} shows that this is also the relative
canonical model of \(Y\) over \(X\).

\subsection{Computing relative canonical models}
\label{subsec:computing-relative-canonical-models}

Lemma \ref{lem:stack-Rees-candidate} reduces the verification of a
candidate to testing that its morphism is small and that its source
satisfies \(S_2\). The first condition is tested as follows. Call a point of
\(W\) \emph{exceptional} for a morphism \(\pi\colon W\to V\) when
\(\cO_{V,\pi(w)}\to\cO_{W,w}\) is not an isomorphism, and write
\(\Exc(\pi)\) for the set of such points.
\begin{lem}
\label{lem:small-test}
Let \(V\) be a normal variety, let \(\cJ\) be a nonzero coherent ideal
sheaf on \(V\) and let
\[
\pi\colon W=\cProj_{V}\bigoplus_{n\ge0}\cJ^{n}\to V
\]
be the blow-up of \(V\) along \(\cJ\), which we assume integral. Then
\(\cJ\cO_{W}\) is invertible and cuts out a closed subscheme of \(W\) of
pure codimension one, and \(\Exc(\pi)\) has codimension at least two if
and only if
\[
\dim\pi(D)=\dim D
\]
for every irreducible component \(D\) of \(V(\cJ\cO_{W})\).
\end{lem}

\begin{proof}
That \(\cJ\cO_{W}\) is invertible is a property of the
blow-up \cite[Chapter II, Proposition 7.13]{hartshorne1977algebraic}.
It is therefore locally generated by one element, which is a nonzero
divisor because \(W\) is integral and \(\cJ\ne0\), so every component of
\(V(\cJ\cO_{W})\) has codimension one by Krull's principal ideal
theorem.

Over \(V\setminus V(\cJ)\) the ideal \(\cJ\) is already invertible and
\(\pi\) is an isomorphism, so \(\Exc(\pi)\subseteq V(\cJ\cO_{W})\). Since
the latter has pure codimension one, the points of codimension one of
\(\Exc(\pi)\) are among the generic points of the components of
\(V(\cJ\cO_{W})\), and \(\Exc(\pi)\) has codimension at least two exactly
when none of those generic points is exceptional.

Let \(D\) be such a component, \(\eta\) its generic point and
\(\xi:=\pi(\eta)\), the generic point of \(\overline{\pi(D)}\). If
\(\dim\pi(D)<\dim D\) then \(\cO_{V,\xi}\) has dimension at least two while
\(\cO_{W,\eta}\) has dimension one, so \(\eta\) is exceptional. If
\(\dim\pi(D)=\dim D\) then \(\xi\) is a point of codimension one of the
normal variety \(V\), so \(\cO_{V,\xi}\) is a discrete valuation ring; it
is dominated by \(\cO_{W,\eta}\) inside their common fraction field, and
a valuation ring is maximal for domination, so the two agree and \(\eta\)
is not exceptional.
\end{proof}

\begin{algorithm}[!ht]
\raggedright
\caption{Computing the relative canonical model}
\label{algo:relative-canonical-model}

\textbf{Input:} A projective birational graph morphism
\(f\colon Y\to X\) of normal monograded varieties such that
\(f_{*}\cO_{Y}=\cO_{X}\), \(Y\) has only log terminal singularities, and
\(-K_{Y}\) is \(f\)-ample.

\textbf{Output:} A graph morphism \(g\colon Z\to X\) representing the
relative canonical model of \(Y\) over \(X\).

\textbf{Procedure:}
\begin{enumerate}
\item Let \(R\) be the homogeneous coordinate ring of \(X\) and let
\(\omega_{R}\) be its canonical module, computed as in Section
\ref{sec:nefness-canonical}. Compute the graded \(R\)-module
\(N:=\Hom_{R}(\omega_{R},R)\), let \(d\) be the least integer with
\(N_{d}\ne0\), take a nonzero \(\phi\in N_{d}\), and put
\(I:=\phi(\omega_{R})\), a homogeneous ideal of \(R\) isomorphic to
\(\omega_{R}(-d)\) as a graded \(R\)-module. (Any homogeneous ideal so
isomorphic for some integer \(d\) serves in the steps that follow; the
least \(d\) is taken because the choice governs their cost.)
\item Put \(m=1\).
\item \label{enu:compute-Rees}%
Take a minimal set of homogeneous generators
\(f_{1},\dots,f_{r}\) of \(I^{(m)}\), of degrees \(d_{1},\dots,d_{r}\). If
\(r=1\), return the identity of \(X\) as a graph morphism and stop.
Otherwise find a representation
\(R[u_{1},\dots,u_{r}]/(h_{1},\dots,h_{l})\) of the Rees algebra
\(R[I^{(m)}t]\), bigraded by \(\deg u_{i}=(d_{i},1)\). It presents the
blow-up
\[
\pi^{m}\colon V(h_{1},\dots,h_{l})\to X
\]
of \(X\) along \(I^{(m)}\), the ambient \(\biProj R[u_{1},\dots,u_{r}]\)
being proper over \(X\) with fiber \(\PP^{r-1}\), and equal to
\(X\times\PP^{r-1}\) when the \(d_{i}\) are all equal.
Compute the minimal primes of
\(I^{(m)}\cO_{V(h_{1},\dots,h_{l})}\), check whether
\(\dim\pi^{m}(D)=\dim D\) for each of the components \(D\) they define
(Lemma \ref{lem:small-test}), and check whether
\(V(h_{1},\dots,h_{l})\) satisfies \(S_{2}\). When these conditions
are satisfied, we transform the obtained bi-to-mono projection
\[
V(h_{1},\dots,h_{l})
=\biProj\bigl(R[u_{1},\dots,u_{r}]/(h_{1},\dots,h_{l})\bigr)
\to\Proj R
\]
into a graph morphism and return it as an output.
Otherwise, put
\(m=m+1\) and redo Step \ref{enu:compute-Rees}.
\end{enumerate}
\end{algorithm}

\begin{prop}
\label{prop:compute-relative-canonical-model}
Let \(f\colon Y\to X\) be a projective birational graph morphism of
normal monograded varieties such that \(f_*\cO_Y=\cO_X\), \(Y\) has only
log terminal singularities, and \(-K_Y\) is \(f\)-ample. Then we can
compute the relative canonical model \(g\colon Z\to X\) as a graph
morphism of monograded varieties by Algorithm
\ref{algo:relative-canonical-model}. If \(f\) is small, this is the
flip of \(f\).
\end{prop}

\begin{proof}
Let \(R\) be the homogeneous coordinate ring of \(X\) and let \(\omega_{R}\)
be its canonical module, computed as in Section
\ref{sec:nefness-canonical}. It is a torsion-free graded \(R\)-module of
rank one by \cite[12.1.9(i), 12.1.18(ii), 14.5.2]{brodmann2012localcohomology}.
Consequently it embeds homogeneously into \(R\) after a shift, so Step 1
computes an ideal \(I\cong\omega_R(-d)\). Its symbolic powers are computed
as double duals; see \cite{drabkin2019calculations}.

For each \(m\), the Rees algebra of \(I^{(m)}\) and its bi-to-mono
projection are computable by
\cite[Proposition 10.2.11]{vasconcelos1994computational}.
Lemma \ref{lem:small-test} tests whether this projection is small. If
\(I^{(m)}\) is principal, its blow-up is the identity of \(X\). In all
cases, every candidate returned by the algorithm is the relative
canonical model by Lemma \ref{lem:stack-Rees-candidate}.

It remains to prove termination. For a sufficiently divisible \(m\), the
line bundle \(\cO_X(md)\) is invertible and the Veronese algebra
\[
\cT^m:=\bigoplus_{n\ge0}\cO_X(nmK_X)
\]
is generated in degree one
\cite[Tag~0EGH]{stacksprojectauthors2022stacksproject}. Equivariant
sheafification of \(I^{(m)}\cong\omega_R^{[m]}(-md)\), followed by the
twist by \(\cO_X(md)\), identifies the candidate with the component
dominating \(X\) in
\[
\cProj_X\cS^m,
\qquad
\cS^m:=\bigoplus_{n\ge0}\cO_X(mK_X)^{\otimes n}.
\]
The natural map \(\cS^m\to\cT^m\) is surjective, and this component is
therefore
\[
\cProj_X\cT^m=Z.
\]
Thus such an \(m\) passes both tests, and
the search terminates with the correct model. Finally, Section
\ref{subsec:bi-to-mono-projections} converts the resulting bi-to-mono
projection into a graph morphism.
\end{proof}

Following \cite[Definition~1]{kollar2021relative}, a step of the minimal
model program with scaling of an ample divisor \(H\) is a diagram
\(X'\to Z'\leftarrow X''\) in which, writing \(H'\) and \(H''\) for
the birational transforms of \(H\) on \(X'\) and \(X''\), the first
morphism is the contraction defined by \(K_{X'}+rH'\) for the
relevant value of \(r\), the second is small, and
\(K_{X''}+(r-\epsilon)H''\) is ample for \(0<\epsilon\ll1\).

\begin{lem}
\label{lem:mixed-step}
Let \(f\colon Y\to X\) be a birational contraction, and let
\(g\colon Z\to X\) be a small birational morphism such that \(K_Z\) is
\(\QQ\)-Cartier and \(g\)-ample. Let \(H\) be an ample Cartier divisor on
\(Y\), let \(\lambda\in\QQ_{>0}\) and suppose that
\(K_{Y}+\lambda H\sim_{\QQ}f^{*}L\) for an ample
\(\QQ\)-Cartier divisor \(L\) on \(X\). Let \(H_{Z}\) be the strict transform
of \(f_{*}H\) on \(Z\). Then \(H_{Z}\) is \(\QQ\)-Cartier and, for
\(0<\epsilon<\lambda\),
\[
K_{Y}+(\lambda-\epsilon)H\sim_{\RR,X}-\epsilon H
\quad\text{is \(f\)-anti-ample},
\]
\[
K_{Z}+(\lambda-\epsilon)H_{Z}\sim_{\RR,X}\tfrac{\epsilon}{\lambda}K_{Z}
\quad\text{is \(g\)-ample},
\]
while for \(0<\epsilon\ll1\) both \(K_{Y}+(\lambda+\epsilon)H\) and
\(K_{Z}+(\lambda-\epsilon)H_{Z}\) are ample.
\end{lem}
\begin{proof}
Subtracting \(\epsilon H\) from the hypothesis
\(K_{Y}+\lambda H\sim_{\QQ}f^{*}L\sim_{\QQ,X}0\) gives
\(K_{Y}+(\lambda-\epsilon)H\sim_{\RR,X}-\epsilon H\), which is
\(f\)-anti-ample because \(H\) is ample.
Pushing \(K_{Y}+\lambda H\sim_{\QQ}f^{*}L\) forward along \(f\) gives
\(K_{X}+\lambda f_{*}H\sim_{\QQ}L\); since \(g\) is small, taking strict
transforms yields
\[
K_{Z}+\lambda H_{Z}\sim_{\QQ}g^{*}L.
\]
Here \(K_{Z}\) and \(g^{*}L\) are \(\QQ\)-Cartier, hence so is \(H_{Z}\).
Over \(X\) we have \(g^{*}L\sim_{\QQ,X}0\), so
\(K_{Z}\sim_{\QQ,X}-\lambda H_{Z}\) and therefore
\[
K_{Z}+(\lambda-\epsilon)H_{Z}\sim_{\RR,X}
K_{Z}-\frac{\lambda-\epsilon}{\lambda}K_{Z}=\frac{\epsilon}{\lambda}K_{Z},
\]
which is \(g\)-ample because \(K_{Z}\) is and \(0<\epsilon<\lambda\).

On \(Y\), the divisor
\(K_{Y}+(\lambda+\epsilon)H\sim_{\RR}f^{*}L+\epsilon H\) is the sum of a
nef and an ample divisor, hence ample, for every \(\epsilon>0\). On
\(Z\), the divisor \(-H_{Z}\sim_{\QQ,X}\frac{1}{\lambda}K_{Z}\) is
\(g\)-ample, so \(-H_{Z}+Ng^{*}L\) is ample for \(N\gg0\). Hence for
\(0<\epsilon\ll1\), the divisor
\[
K_{Z}+(\lambda-\epsilon)H_{Z}\sim_{\RR}g^{*}L-\epsilon H_{Z}
=\epsilon\Bigl(-H_{Z}+\frac{1}{\epsilon}g^{*}L\Bigr)
\]
is ample as well.
\end{proof}

\begin{cor}
\label{cor:relative-canonical-model-MMP-step}
Let \(f\colon Y\to X\) be the contraction produced by Algorithm
\ref{algo:find-contraction} from an ample divisor \(H\) at the
threshold \(\lambda\), and let \(g\colon Z\to X\) be the relative
canonical model of \(f\) computed by Algorithm
\ref{algo:relative-canonical-model}, as in Proposition
\ref{prop:compute-relative-canonical-model}. Then
\(Y\to X\leftarrow Z\) is a step of the \(K\)-MMP in the sense of
\cite[Definition~3.5]{hashizume2025normal}. Viewed as a one-step sequence, it
is a \(K\)-MMP with scaling of \(H\) in the sense of the same
definition. It is also a step of the minimal model program with scaling
of \(H\) in the sense of \cite[Definition~1]{kollar2021relative}.
\end{cor}

\begin{proof}
By Proposition \ref{prop:compute-relative-canonical-model}, the divisor
\(-K_Y\) is \(f\)-ample, the morphism \(g\) is small, and \(K_Z\) is
\(g\)-ample. Thus the diagram is a step of the \(K\)-MMP in the sense
of Hashizume.

We produce the divisor \(L\) that Lemma \ref{lem:mixed-step} asks for.
Write
\[
Y\xrightarrow{f}X\xrightarrow{\nu}\PP^{N}
\]
for the Stein factorization of
\(\Phi_{|M(K_{Y}+\lambda H)|}\) constructed in Proposition
\ref{prop:contraction-from-lambda}. The divisor
\(A:=\nu^{*}\cO_{\PP^{N}}(1)\) is ample Cartier and
\[
\cO_{Y}\bigl(M(K_{Y}+\lambda H)\bigr)\cong f^{*}A.
\]
Taking \(L:=(1/M)A\) gives
\(K_{Y}+\lambda H\sim_{\QQ}f^{*}L\), so Lemma
\ref{lem:mixed-step} applies.

In particular, the birational transform \(H_Z\) is \(\QQ\)-Cartier and
\(K_Z+(\lambda-\epsilon)H_Z\) is ample for
\(0<\epsilon\ll1\), so the nef threshold on \(Z\) is well defined. On
the other hand, by the definition of \(\lambda\), the divisor
\(K_Y+\lambda H\) is nef and has intersection number zero with every
curve contracted by \(f\). Thus the one-step sequence is a \(K\)-MMP
with scaling of \(H\) in the sense of Hashizume. The ampleness assertions
of Lemma \ref{lem:mixed-step} also verify Kollár's conditions.
\end{proof}

\subsection{The three-dimensional minimal model program}
\label{subsec:three-dimensional-MMP}

The preceding results assemble into Algorithm \ref{algo:MMP}. At each
iteration, Section \ref{sec:nefness-canonical} either certifies that the
canonical divisor is nef or computes its next contraction. A contraction
of fiber type ends the program, while for a birational contraction
Algorithm \ref{algo:relative-canonical-model} computes the relative
canonical model; Corollary \ref{cor:relative-canonical-model-MMP-step}
shows that the resulting diagram is a step of the minimal model program.
If the target is \(\QQ\)-Gorenstein, its relative canonical model is the
target itself, which is retained without replacing it by a relative
\(\Proj\) presentation. Thus the only remaining point is that the
procedure terminates, as proved in Proposition \ref{prop:termination}
below.

\begin{algorithm}[H]
\raggedright
\caption{The minimal model program for threefolds}
\label{algo:MMP}

\textbf{Input:} A normal monograded variety \(X\) of dimension 3 with
only log terminal singularities.

\textbf{Output:} A finite chain \(X=X_0,X_1,\dots,X_n\) of monograded
varieties describing a minimal model program for \(X\), together with
graph morphisms \(f_i\) connecting consecutive terms. Each \(f_i\) is
either a contraction
\[
f_i\colon X_{i-1}\to X_i
\]
or, when \(X_{i-1}\) is not \(\QQ\)-Gorenstein, its relative canonical
model
\[
f_i\colon X_i\to X_{i-1}.
\]
The chain ends either with a variety whose canonical divisor is nef or
with a \(K\)-negative fibration.

\textbf{Procedure:}
\begin{enumerate}
\item Set \(i=0\) and \(X_{0}\) to be \(X\).
\item \label{enu:algo-step}%
Find a positive integer \(a\) such that \(aK_{X_i}\) is Cartier, as in the
proof of Proposition \ref{prop:nef}, and check whether \(K_{X_i}\) is nef.
If it is, stop and return the rational maps obtained so far.
\item
Using \(a\) as the Cartier-index input to Algorithm
\ref{algo:find-contraction}, compute an ample Cartier divisor \(H\) on
\(X_i\) as in Remark \ref{rem:choice-ample-divisor}, the threshold
\(\lambda\), and the contraction \(\phi\colon X_i\to W\) it defines. Set
\(X_{i+1}:=W\) and \(f_{i+1}:=\phi\).
\item If \(\dim X_{i+1}<\dim X_{i}\), stop and return the chain obtained so far.
\item Otherwise compute the relative canonical model \(g\colon Z\to X_{i+1}\)
of \(f_{i+1}\) (Algorithm \ref{algo:relative-canonical-model}).
Test whether \(g\) is an isomorphism by applying the graph-morphism test of
\cite[Lemma~3.3]{yasuda2023theisomorphism} to its graph.
If it is, put \(i=i+1\) and go to Step \ref{enu:algo-step}. If it is not,
set \(X_{i+2}:=Z\) and
\(f_{i+2}:=g\colon X_{i+2}\to X_{i+1}\); put \(i=i+2\) and go to
Step \ref{enu:algo-step}.
\end{enumerate}
\end{algorithm}

\begin{prop}
\label{prop:termination}
Let \(X\) be a normal projective variety of dimension three with only
log terminal singularities. Then the sequence of steps that Algorithm
\ref{algo:MMP} performs on \(X\) is finite.
\end{prop}

\begin{proof}
By Proposition \ref{prop:compute-relative-canonical-model}, every
birational iteration of Algorithm \ref{algo:MMP} is a diagram
\[
Y_i\to W_i\leftarrow Y_{i+1}
\]
in which \(Y_{i+1}\) is the relative canonical model of \(Y_i\) over
\(W_i\). Thus \(-K_{Y_i}\) is ample over \(W_i\), the morphism
\(Y_{i+1}\to W_i\) is small, and \(K_{Y_{i+1}}\) is ample over \(W_i\).
Consequently this diagram is a step of the \(K\)-MMP in the sense of
\cite[Definition~3.5]{hashizume2025normal}. This definition does not require
\(Y_i/W_i\) to have relative Picard number one; see
\cite[Remark~3.6]{hashizume2025normal}.

Take a crepant \(\QQ\)-factorial modification \(Y'_0\to Y_0\), and
apply successively the lifting construction in the proof of
\cite[Theorem~3.9]{hashizume2025normal}; see also
\cite[Remark~3.10]{hashizume2025normal}. Since every \(Y_i\) is klt, its
non-lc locus is empty. Each lifted block therefore consists of ordinary
MMP steps whose contractions have relative Picard number one. By the
uniqueness of the relative canonical model, its last model maps
crepantly to \(Y_{i+1}\), and is therefore a crepant \(\QQ\)-factorial
modification of \(Y_{i+1}\).
The blocks concatenate, and each is nonempty because \(K_{Y_i}\) is not
nef over \(W_i\). Hence an infinite run of the algorithm would lift to
an infinite classical MMP sequence on \(\QQ\)-factorial klt threefolds, which is
impossible by
\cite{kawamata1992termination,shokurov1996threefold,kollar1992flips}.
\end{proof}

\begin{rem}
When \(K_X\) is pseudo-effective, the termination assertion is already
covered by \cite[Remark~3.16]{hashizume2025normal}. The proof above is essentially the
same lifting argument used there, based on
\cite[Theorem~3.9]{hashizume2025normal}.
\end{rem}

\begin{rem}
\label{rem:implementation}
Prototype Macaulay2 implementations are available online: one for
Stein factorization (Algorithm \ref{algo:Stein})%
\footnote{\url{https://github.com/takehikoyasuda/SteinFactorizationM2}},
one for relative canonical models (Algorithm
\ref{algo:relative-canonical-model})%
\footnote{\url{https://github.com/takehikoyasuda/flip-computation}},
and one driving Algorithm \ref{algo:MMP}%
\footnote{\url{https://github.com/takehikoyasuda/Coding_MMP}}.
The code
was produced through conversations with large language models,
Claude and ChatGPT, rather than written directly by the
author, who has not checked it line by line. It should therefore be
regarded as experimental.

Each basic operation used by the algorithms has been carried out on
very simple examples, including tests of nefness, the computation of a
nef threshold and its contraction, a relative canonical model, and a
Stein factorization. These calculations show that the procedures can
be implemented, but they do not yet amount to an automatic computation
of an interesting MMP sequence. Substantial improvements in efficiency
will be needed for that purpose. Developing such improvements is an
important problem for future work and is well worth pursuing.
\end{rem}

\bibliographystyle{alpha}
\bibliography{AlgoMMP}

\end{document}